\documentclass[12pt,twoside,a4paper]{article}
\usepackage[utf8]{inputenc}
\usepackage[T1]{fontenc}
\usepackage{lmodern}
\usepackage{amsfonts}
\usepackage{amssymb}
\usepackage{amsthm}
\usepackage{graphicx}
\usepackage[bookmarksopen=false,breaklinks=true,%
      backref=page,pagebackref=true,plainpages=false,%
      hyperindex=true,pdfstartview=FitH,colorlinks=true,%
      pdfpagelabels=true,linkcolor=blue,%
      citecolor=red,urlcolor=red,hypertexnames=false%
      ]%
   {hyperref}

\usepackage{mathrsfs}
\usepackage{mathtools}
\usepackage{stmaryrd}
\usepackage[all]{xy}
\usepackage{makeidx}
\usepackage{natbib}
\bibpunct{(}{)}{,}{a}{}{,}
\usepackage[french]{babel}
\usepackage{tocbibind}
\usepackage{xspace}
\usepackage{tikz}
\usepackage{fancyhdr} 
\usepackage{datetime} 

\usepackage{etoolbox}
\makeatletter\def\th@plain{\slshape}
\patchcmd{\th@remark}{\itshape}{\slshape}{}{}\makeatother

\newcounter{bidon}
\newcommand{\rdb}{\refstepcounter{bidon}}

\DeclareMathAlphabet{\mathpzc}{OT1}{pzc}{m}{it}

\patchcmd{\sectionmark}{\MakeUppercase}{}{}{}
\begin{document}


\newtheorem{theorem}{Théorème}[section]
\newtheorem{theoreme}{Théorème}[subsection]
\newtheorem{thdef}[theorem]{Théorème et définition}
\newtheorem{valsatz}[theorem]{Valuativstellensatz}
\newtheorem{vast}[theorem]{\vst}
\newtheorem{vastf}[theorem]{\vst formel}
\newtheorem{pstf}[theorem]{Positivstellensatz formel}
\newtheorem{pst}[theorem]{Positivstellensatz}
\newtheorem{lemma}[theorem]{Lemme}
\newtheorem{corollary}[theorem]{Corolaire}
\newtheorem{conjecture}[theorem]{Conjecture}
\newtheorem{proposition}[theorem]{Proposition}
\newtheorem{prpta}[theorem]{Propriétés attendues}
\newtheorem{propdef}[theorem]{Proposition et définition}
\newtheorem{fact}[theorem]{Fait}
\newtheorem{calculs}[theorem]{Calculs}
\newtheorem{convention}[theorem]{Convention}
\newtheorem{plcc}[theorem]{Principe local-global concret}
\newtheorem{precision}[theorem]{Précision}

\newtheorem{theoremc}[theorem]{Théorème\etoz}
\newtheorem{lemmac}[theorem]{Lemme\etoz}
\newtheorem{corollaryc}[theorem]{Corolaire\etoz}
\newtheorem{proprietec}[theorem]{Propriété\etoz}
\newtheorem{propositionc}[theorem]{Proposition\etoz}
\newtheorem{factc}[theorem]{Fait\etoz}
\newtheorem{propdefc}[theorem]{Proposition et définition\etoz}

\newtheorem{atheorem}{Théorème}[section]
\newtheorem{alemma}[atheorem]{Lemme}
\newtheorem{acorollary}[atheorem]{Corolaire}
\newtheorem{apropriete}[atheorem]{Propriété}
\newtheorem{aproposition}[atheorem]{Proposition}
\newtheorem{apropdef}[atheorem]{Proposition et définition}
\newtheorem{afact}[atheorem]{Fait}

\theoremstyle{definition}
\newtheorem{definition}[theorem]{Définition}
\newtheorem{dfni}[theorem]{Définition informelle}
\newtheorem{definitions}[theorem]{Définitions}
\newtheorem{example}[theorem]{Exemple}
\newtheorem{examples}[theorem]{Exemples}
\newtheorem{notation}[theorem]{Notation}
\newtheorem{ter}[theorem]{Terminologie}
\newtheorem{problem}[theorem]{Problème}
\newtheorem{question}[theorem]{Question}
\newtheorem{questions}[theorem]{Questions}
\newtheorem{context}[theorem]{Contexte}
\newtheorem{descri}[theorem]{Description}

\newtheorem{definitionc}[theorem]{Définition\etoz}
\newtheorem{definota}[theorem]{Définition et notation}
\newtheorem{aquestion}[atheorem]{Question}
\newtheorem{adefinition}[atheorem]{Définition}

\theoremstyle{remark}
\newtheorem{note}[theorem]{Note}
\newtheorem{notes}[theorem]{Notes}
\newtheorem{remark}[theorem]{Remarque}
\newtheorem{remarks}[theorem]{Remarques}
\newtheorem{comment}[theorem]{Commentaire}
\newtheorem{aremark}[atheorem]{Remarque}
\newtheorem{aremarks}[atheorem]{Remarques}

\newcommand{\vou}{\MA{\tsbf{ ou }}}
\newcommand{\Vou}{\MA{\tsbf{OU}}}
\newcommand \EXists[1] {\tsbf{Introduire }{#1}\tsbf{ tel que }\,}
\newcommand \vet {\tsbf{,}\;}
\newcommand \Atcl {\mathrm{Atcl}}



\newcounter{MF}
\newcommand\stMF{\stepcounter{MF}}

\newcommand{\lec}{\stMF\ifodd\value{MF}lecteur\xspace\else 
lectrice\xspace\fi}

\newcommand{\lecs}{\stMF\ifodd\value{MF}lecteurs\xspace\else 
lectrices\xspace\fi}

\newcommand{\alec}{\stMF\ifodd\value{MF}au lecteur\xspace\else%
à la lectrice\xspace\fi}

\newcommand{\dlec}{\stMF\ifodd\value{MF}du lecteur\xspace\else%
de la lectrice\xspace\fi}

\newcommand{\llec}{\stMF\ifodd\value{MF}le lecteur\xspace\else la lectrice\xspace\fi}

\newcommand{\Llec}{\stMF\ifodd\value{MF}Le lecteur\xspace\else La lectrice\xspace\fi}

\newcommand{\lui}{\ifodd\value{MF}lui\xspace\else
elle\xspace\fi}

\newcommand{\celui}{\ifodd\value{MF}celui\xspace\else
celle\xspace\fi}

\newcommand{\ceux}{\ifodd\value{MF}ceux\xspace\else
celles\xspace\fi}

\newcommand{\er}{\ifodd\value{MF}er\xspace\else
ère\xspace\fi}

\newcommand{\eux}{\ifodd\value{MF}eux\xspace\else
elles\xspace\fi}

\newcommand{\eUx}{\ifodd\value{MF}eux\xspace\else
euse\xspace\fi}

\newcommand{\eUX}{\ifodd\value{MF}eux\xspace\else
euses\xspace\fi}

\newcommand{\leux}{\ifodd\value{MF}leux\xspace\else
leuse\xspace\fi}

\newcommand{\il}{\ifodd\value{MF}il\xspace\else
elle\xspace\fi}

\newcommand{\ien}{\ifodd\value{MF}ien\xspace\else
ienne\xspace\fi}

\newcommand{\e}{\ifodd\value{MF}\xspace \else e\xspace\fi}

\newcommand{\n}{\ifodd\value{MF}n\xspace\else nne\xspace\fi}

\makeatletter
\newcommand{\la}{\@ifstar{\ifodd\value{MF}le\else
la\fi}{\stMF\ifodd\value{MF}le\else la\fi}}
\makeatother

\newcommand \Calculs[1]{
\begin{proof}[\textsl{Calculs}]
#1
\end{proof}
}

\newcommand \Note{\rdb
\noi{\sl Note. }}

\newcommand \rem{\rdb
\noi{\sl Remarque. }}

\newcommand \REM[1]{\rdb
\noi{\sl Remarque#1. }}

\newcommand \rems{\rdb
\noi{\sl Remarques. }}

\newcommand \exl{\rdb
\noi{\bf Exemple. }}

\newcommand \EXL[1]{\rdb
\noi{\bf Exemple: #1. }}

\newcommand \exls{\rdb
\noi{\bf Exemples. }}

\newcommand \thref[1] {théorème~\ref{#1}}
\newcommand \paref[1] {page~\pageref{#1}}
\newcommand \pstfref[1] {Positivstellensatz formel~\ref{#1}}
\newcommand \pstref[1] {Positivstellensatz~\ref{#1}}

\newcommand\oge{\leavevmode\raise.3ex\hbox{$\scriptscriptstyle\langle\!\langle\,$}}
\newcommand\feg{\leavevmode\raise.3ex\hbox{$\scriptscriptstyle\,\rangle\!\rangle$}}

\newcommand\gui[1]{\oge{#1}\feg}

\newcommand \facile{\begin{proof}
La démonstration est laissée \alec.
\end{proof}
}

\def \num {\no} 

\newcommand\comm{\rdb
\noi{\sl Commentaire. }}

\newcommand\COM[1]{\rdb
\noi{\sl Commentaire #1. }}

\newcommand\comms{\rdb
\noi{\sl Commentaires. }}

\newcommand\Subsubsection[1]{
\rdb\addcontentsline{toc}{subsubsection}{#1} \subsubsection*{#1}}

\newcommand\Subsection[1]{
\rdb\addcontentsline{toc}{subsection}{#1} \subsection*{#1}}

\newcommand\Section[1]{
\rdb\addcontentsline{toc}{section}{#1} \section*{#1}}

\newcommand\Intro{
\rdb\addcontentsline{toc}{section}{Introduction} \section*{Introduction}}

\renewcommand\paragraph[1]{
\rdb\addcontentsline{toc}{subsubsection}{#1} 

\medskip \noindent $\bullet$ \textbf{#1}}

\newcommand\Pb{\rdb
\noi{\bf Problème. }}

\newcommand\eoq{\hbox{}\nobreak
\vrule width 1.4mm height 1.4mm depth 0mm}

\newcommand \Cad {C'est-à-dire\xspace}
\newcommand \recu {récurrence\xspace}
\newcommand \hdr {hypothèse de \recu}
\newcommand \cad {c'est-à-dire\xspace}
\newcommand \cade {c'est-à-dire encore\xspace}
\newcommand \ssi {si, et seulement si, }
\newcommand \ssiz {si, et seulement si,~}
\newcommand \cnes {condition nécessaire et suffisante\xspace}
\newcommand \spdg {sans perte de généralité\xspace}
\newcommand \Spdg {Sans perte de généralité\xspace}

\newcommand \Propeq {Les propriétés suivantes sont 
équivalentes.}
\newcommand \propeq {les propriétés suivantes sont 
équivalentes.}

\newcommand \Kev {$\gK$-\evc}
\newcommand \Kbev {$\gKb$-\evc}
\newcommand \Kevs {$\gK$-\evcs}

\newcommand \Lev {$\gL$-\evc}
\newcommand \Levs {$\gL$-\evcs}

\newcommand \Qev {$\QQ$-\evc}
\newcommand \Qevs {$\QQ$-\evcs}

\newcommand \kev {$\gk$-\evc}
\newcommand \kevs {$\gk$-\evcs}

\newcommand \lev {$\gl$-\evc}
\newcommand \levs {$\gl$-\evcs}

\newcommand \Alg {$\gA$-\alg}
\newcommand \Algs {$\gA$-\algs}

\newcommand \Blg {$\gB$-\alg}
\newcommand \Blgs {$\gB$-\algs}

\newcommand \Clg {$\gC$-\alg}
\newcommand \Clgs {$\gC$-\algs}

\newcommand \klg {$\gk$-\alg}
\newcommand \klgs {$\gk$-\algs}

\newcommand \llg {$\gl$-\alg}
\newcommand \llgs {$\gl$-\algs}

\newcommand \Klg {$\gK$-\alg}
\newcommand \Klgs {$\gK$-\algs}

\newcommand \Llg {$\gL$-\alg}
\newcommand \Llgs {$\gL$-\algs}

\newcommand \QQlg {$\QQ$-\alg}
\newcommand \QQlgs {$\QQ$-\algs}

\newcommand \Rlg {$\gR$-\alg}
\newcommand \Rlgs {$\gR$-\algs}

\newcommand \RRlg {$\RR$-\alg}
\newcommand \RRlgs {$\RR$-\algs}

\newcommand \Vlg {$\gV$-\alg}
\newcommand \Vlgs {$\gV$-\algs}

\newcommand \ZZlg {$\ZZ$-\alg}
\newcommand \ZZlgs {$\ZZ$-\algs}

\newcommand \Amo {$\gA$-module\xspace}
\newcommand \Amos {$\gA$-modules\xspace}

\newcommand \Bmo {$\gB$-module\xspace}
\newcommand \Bmos {$\gB$-modules\xspace}

\newcommand \Cmo {$\gC$-module\xspace}
\newcommand \Cmos {$\gC$-modules\xspace}

\newcommand \kmo {$\gk$-module\xspace}
\newcommand \kmos {$\gk$-modules\xspace}

\newcommand \Kmo {$\gK$-module\xspace}
\newcommand \Kmos {$\gK$-modules\xspace}

\newcommand \Lmo {$\gL$-module\xspace}
\newcommand \Lmos {$\gL$-modules\xspace}

\newcommand \Vmo {$\gV$-module\xspace}
\newcommand \Vmos {$\gV$-modules\xspace}

\newcommand \Zmo {$\gZ$-module\xspace}
\newcommand \Zmos {$\gZ$-modules\xspace}

\newcommand \ZZmo {$\ZZ$-module\xspace}
\newcommand \ZZmos {$\ZZ$-modules\xspace}

\newcommand \Ali {application $\gA$-\lin}
\newcommand \Alis {applications $\gA$-\lins}

\newcommand \kli {application $\gk$-\lin}
\newcommand \klis {applications $\gk$-\lins}

\newcommand \Kli {application $\gK$-\lin}
\newcommand \Klis {applications $\gK$-\lins}

\newcommand \Bli {application $\gB$-\lin}
\newcommand \Blis {applications $\gB$-\lins}

\newcommand \Cli {application $\gC$-\lin}
\newcommand \Clis {applications $\gC$-\lins}

\newcommand \Vli {application $\gV$-\lin}
\newcommand \Vlis {applications $\gV$-\lins}

\newcommand \ac{algébriquement clos\xspace}  

\newcommand \acl {anneau \icl}
\newcommand \acls {anneaux \icl}

\newcommand \adp {anneau de Pr\"ufer\xspace}
\newcommand \adps {anneaux de Pr\"ufer\xspace}

\newcommand \adpc {\adp \coh}
\newcommand \adpcs {\adps \cohs}

\newcommand \adu {\alg de décomposition universelle\xspace}
\newcommand \adus {\algs de décomposition universelle\xspace}

\newcommand \adv {anneau de valuation\xspace}
\newcommand \advs {anneaux de valuation\xspace}

\newcommand \advd {anneau de valuation discrète\xspace}
\newcommand \advds {anneaux de valuation discrète\xspace}

\newcommand \advl {anneau \dvla} 
\newcommand \advls {anneaux \dvlas} 

\newcommand \Afr {Anneau \frl}
\newcommand \Afrs {Anneaux \frls}
\newcommand \afr {anneau \frl}
\newcommand \afrs {anneaux \frls}

\newcommand \aFr {\hyperref[theorieAfr]{anneau \frl}}
\newcommand \aFrs {\hyperref[theorieAfr]{anneau \frls}}

\newcommand \afrr {\afr réduit\xspace}
\newcommand \afrrs {\afrs réduits\xspace}
\newcommand \Afrrs {\Afrs réduits\xspace}

\newcommand \afrvr {\afr avec \ravs}
\newcommand \aFrvr {\hyperref[theorieAfrrv]{\afrvr}}
\newcommand \afrvrs {\afrs avec \ravs}

\newcommand \aftr {anneau réticulé \ftm réel\xspace}
\newcommand \aftrs {anneaux réticulés \ftm réels\xspace}

\newcommand \aG {\alg galoisienne\xspace}
\newcommand \aGs {\algs galoisiennes\xspace}

\newcommand \agB {\alg de Boole\xspace}
\newcommand \agBs {\algs de Boole\xspace}

\newcommand \agH {\alg de Heyting\xspace}
\newcommand \agHs {\algs de Heyting\xspace}

\newcommand \agq{algébrique\xspace}
\newcommand \agqs{algébriques\xspace}

\newcommand \agqt{algébriquement\xspace}

\newcommand \aKr {anneau de Krull\xspace}
\newcommand \aKrs {anneaux de Krull\xspace}

\newcommand \ale {\alg étale\xspace}
\newcommand \ales {\algs étales\xspace}

\newcommand \alg {algèbre\xspace}
\newcommand \algs {algèbres\xspace}

\newcommand \algo{algorithme\xspace}
\newcommand \algos{algorithmes\xspace}

\newcommand \algq{algorithmique\xspace}
\newcommand \algqs{algorithmiques\xspace}

\newcommand \alh{\alo \hen}
\newcommand \alhs{\alos \hens}

\newcommand \ali {application \lin}
\newcommand \alis {applications \lins}

\newcommand \algb {anneau \lgb}
\newcommand \algbs {anneaux \lgbs}

\newcommand \alo {anneau local\xspace}
\newcommand \alos {anneaux locaux\xspace}

\newcommand \alrd {\alo \dcd}
\newcommand \alrds {\alos \dcds}

\newcommand \alrdh{\alrd \hen}
\newcommand \alrdhs{\alrds \hens}

\newcommand \anar {anneau \ari}
\newcommand \anars {anneaux \aris}

\newcommand \anor {anneau normal\xspace}
\newcommand \anors {anneaux normaux\xspace}

\newcommand \apf {\alg \pf}
\newcommand \apfs {\algs \pf}

\newcommand \apG {\alg prégaloisienne\xspace}
\newcommand \apGs {\algs prégaloisiennes\xspace}

\newcommand \arch {archimédien\xspace}
\newcommand \arche {archimédienne\xspace}
\newcommand \archs {archimédiens\xspace}
\newcommand \arches {archimédiennes\xspace}

\newcommand \arc {anneau réel clos\xspace}
\newcommand \aRc {\hyperref[theorieArc]{\arc}}
\newcommand \arcs {anneaux réels clos\xspace}

\newcommand \ari{arithmétique\xspace}  
\newcommand \aris{arithmétiques\xspace}  

\newcommand \arv {\adv}
\newcommand \arvs {\advs}

\newcommand \Asr {Anneau \str}
\newcommand \Asrs {Anneaux \strs}
\newcommand \asr {anneau \str}
\newcommand \asrs {anneaux \strs}

\newcommand \asrvr {\asr avec \ravs}
\newcommand \asrvrs {\asrs avec \ravs}

\newcommand \atf {\alg \tf}
\newcommand \atfs {\algs \tf}

\newcommand \atfr {anneau total de fractions\xspace}
\newcommand \atfrs {anneaux totals de fractions\xspace}

\newcommand \auto {automorphisme\xspace}
\newcommand \autos {automorphismes\xspace}


\newcommand \bdg {base de Gr\"obner\xspace}
\newcommand \bdgs {bases de Gr\"obner\xspace}

\newcommand \bdp {base de \dcn partielle\xspace}
\newcommand \bdps {bases de \dcn partielle\xspace}

\newcommand \bdf {base de \fap}

\newcommand \Bif {Borne inférieure\xspace} %
\newcommand \bif {borne inférieure\xspace} %
\newcommand \bifs {bornes inférieures\xspace} %

\newcommand \bsp {borne supérieure\xspace} %
\newcommand \bsps {borne supérieures\xspace} %


\newcommand \cac{corps \ac}  

\newcommand \calf{calcul formel\xspace}  

\newcommand \cara{caractéristique\xspace}  
\newcommand \caras{caractéristiques\xspace}  

\newcommand \carn{caractérisation\xspace}  
\newcommand \carns{caractérisations\xspace}  

\newcommand \carar{caractériser\xspace}

\newcommand \carf{de caractère fini\xspace}  

\newcommand \cdac{\cdi \ac}  
\newcommand \cdacs{\cdis \ac}  
\newcommand \cdi{corps discret\xspace}
\newcommand \cdis{corps discrets\xspace}

\newcommand \cdf{corps de fractions\xspace}

\newcommand \cdH{code de Hensel\xspace}
\newcommand \cdHs{codes de Hensel\xspace}

\newcommand \cdr{corps de racines\xspace}
  
\newcommand \cdv{changement de variables\xspace}  
\newcommand \cdvs{changements de variables\xspace}

\newcommand \cla {clôture \agq}
\newcommand \clas {clôtures \agqs}

\newcommand \cli {clôture intégrale\xspace}
\newcommand \clis {clôtures intégrales\xspace}

\newcommand \clr {clôture réelle\xspace}
\newcommand \clrs {clôtures réelles\xspace}

\newcommand \clsep {clôture \spl}
\newcommand \clseps {clôtures \spls}

\newcommand \codi {corps ordonné discret\xspace}
\newcommand \codis {corps ordonnés discrets\xspace}

\newcommand \coe {coefficient\xspace}
\newcommand \coes {coefficients\xspace}

\newcommand \coh {cohérent\xspace}
\newcommand \cohs {cohérents\xspace}

\newcommand \cohc {cohérence\xspace}

\newcommand \colo {couple local\xspace}
\newcommand \colos {couples locaux\xspace}

\newcommand \colH {\colo hensélien\xspace}
\newcommand \colHs {\colos henséliens\xspace}

\newcommand \coli {combinaison \lin}
\newcommand \colis {combinaisons \lins}

\newcommand \com {comaximaux\xspace}
\newcommand \come {comaximales\xspace}

\newcommand \coo {coordonnée\xspace}
\newcommand \coos {coordonnées\xspace}

\newcommand \cop {complémentaire\xspace}
\newcommand \cops {complémentaires\xspace}

\newcommand \corl {corolaire\xspace}
\newcommand \corls {corolaires\xspace}

\newcommand \cost {construction\xspace}
\newcommand \costs {constructions\xspace}

\newcommand \cosv {conservative\xspace}
\newcommand \cosvs {conservatives\xspace}

\newcommand \cOsv {\hyperref[defithconserv]{conservative\xspace}}
\newcommand \cOsvs {\hyperref[defithconserv]{conservatives\xspace}}

\newcommand \covr {corps ordonné avec \ravs}
\newcommand \covrs {corps ordonnés avec \ravs}

\newcommand \cpb {compatible\xspace} 
\newcommand \cpbs {compatibles\xspace} 

\newcommand \cpbt {compatibilité\xspace} 
\newcommand \cpbtz {compatibilité} 

\newcommand \crdl {corps résiduel\xspace}
\newcommand \crdls {corps résiduels\xspace}

\newcommand \crc {corps réel clos\xspace}
\newcommand \crcs {corps réels clos\xspace}

\newcommand \crcd {corps réel clos discret\xspace}
\newcommand \crcds {corps réels clos discrets\xspace}

\newcommand \cval{corps valué\xspace}
\newcommand \cvals{corps valués\xspace}

\newcommand \cvar{corps valorisé\xspace}
\newcommand \cvars{corps valorisés\xspace}

\newcommand \cvd{\cval discret\xspace}
\newcommand \cvds{\cvals discrets\xspace}
\newcommand \Cvd{Corps valué discret\xspace}
\newcommand \Cvds{Corps valués discrets\xspace}

\newcommand \cvdac{\cvd\ac}
\newcommand \cvdacs{\cvds\ac}

\newcommand \cvdh{\cvd hensélien\xspace}
\newcommand \cvdhs{\cvds henséliens\xspace}

\newcommand \cvdsc{\cvd \splc}
\newcommand \cvdscs{\cvds \splc}

\newcommand \cvul{corps \ultm}
\newcommand \cvuls{corps \ultms}

\newcommand \cvud{corps \ultd}
\newcommand \cvuds{corps \ultds}


\newcommand \dcd {\rdt discret\xspace}
\newcommand \dcde {\rdt discrète\xspace}
\newcommand \dcds {\rdt discrets\xspace}

\newcommand \dcn {décomposition\xspace}
\newcommand \dcns {décompositions\xspace}

\newcommand \dcnb {\dcn bornée\xspace}

\newcommand \dcnc {\dcn complète\xspace}

\newcommand \dcnp {\dcn partielle\xspace}

\newcommand \dcp {décomposable\xspace}
\newcommand \dcps {décomposables\xspace}

\newcommand \ddk {dimension de~Krull\xspace}
\newcommand \ddi {de dimension inférieure ou égale à~}

\newcommand \dimm {description immédiate\xspace}
\newcommand \dimms {descriptions immédiates\xspace}

\newcommand \ddp {domaine de Pr\"ufer\xspace}
\newcommand \ddps {domaines de Pr\"ufer\xspace}

\newcommand \Demo{Démonstration\xspace}     
\newcommand \Demos{Démonstrations\xspace}     

\newcommand \demo{démonstration\xspace}     
\newcommand \demos{démonstrations\xspace}     

\newcommand \dems{démonstrations\xspace}

\newcommand \deno{dénominateur\xspace}     
\newcommand \denos{dénominateurs\xspace}   

\newcommand \deter {déterminant\xspace}  
\newcommand \deters {déterminants\xspace}  
  
\newcommand \Dfn{Définition\xspace}  
\newcommand \Dfns{Définitions\xspace}  
\newcommand \dfn{définition\xspace}  
\newcommand \dfns{définitions\xspace}  

\newcommand \dftr {droite réticulée \ftm réelle\xspace}
\newcommand \dftrs {droites réticulées \ftm réelles\xspace}
  
\newcommand \dil{différentiel\xspace}  
\newcommand \dils{différentiels\xspace}  
\newcommand \dile{différentielle\xspace}  
\newcommand \diles{différentielles\xspace}  

\newcommand \dip{diviseur principal\xspace}
\newcommand \dips{diviseurs principaux\xspace}

\newcommand \discri{discriminant\xspace}  
\newcommand \discris{discriminants\xspace}  

\newcommand \divle {dimension divisorielle\xspace}

\newcommand \dit{distributivité\xspace}

\newcommand \dlg{d'élargissement\xspace}  

\newcommand \dok {domaine de Dedekind\xspace}
\newcommand \doks {domaines de Dedekind\xspace}

\newcommand \dvla {à diviseurs\xspace}
\newcommand \dvlas {à diviseurs\xspace}

\newcommand \dvld {\dvlt décomposé\xspace} %
\newcommand \dvlds {\dvlt décomposés\xspace} %

\newcommand \dvn {dérivation\xspace}
\newcommand \dvns {dérivations\xspace}

\newcommand \dvlg {divisoriel\xspace} 
\newcommand \dvlgs {divisoriels\xspace} 

\newcommand \dvli {\dvlt inversible\xspace} 
\newcommand \dvlis {\dvlt inversibles\xspace} 

\newcommand \dvlt {divisoriellement\xspace} %

\newcommand \dvz {diviseur de zéro\xspace}
\newcommand \dvzs {diviseurs de zéro\xspace}

\newcommand \dve {divisibilité\xspace}

\newcommand \dvee {à \dve explicite\xspace}

\newcommand \dvr {diviseur\xspace}
\newcommand \dvrs {diviseurs\xspace}


\newcommand \Eds {Extension des scalaires\xspace}
\newcommand \edss {extensions des scalaires\xspace}
\newcommand \eds {extension des scalaires\xspace}

\newcommand \eco {\elts \com}

\newcommand \egmt {également\xspace}

\newcommand \egt {égalité\xspace}
\newcommand \egts {égalités\xspace}

\newcommand \eimm {extension immédiate\xspace}
\newcommand \eimms {extensions immédiates\xspace}

\newcommand \eli{élimination\xspace}  

\newcommand \elr{élémentaire\xspace}  
\newcommand \elrs{élémentaires\xspace}  

\newcommand \elrt{élémentairement\xspace}  

\newcommand \elt{élément\xspace}  
\newcommand \elts{éléments\xspace}  

\def \endo {endomorphisme\xspace}
\def \endos {endomorphismes\xspace}

\newcommand \entrel {relation implicative\xspace}
\newcommand \entrels {relations implicatives\xspace}

\newcommand\evc{espace vectoriel\xspace} 
\newcommand\evcs{espaces vectoriels\xspace} 

\newcommand \eqn {équation\xspace}  
\newcommand \eqns {équations\xspace}  

\newcommand \eqv {équivalent\xspace}  
\newcommand \eqve {équivalente\xspace}  
\newcommand \eqvs {équivalents\xspace}  
\newcommand \eqves {équivalentes\xspace}  

\newcommand \eqvc {équivalence\xspace}  
\newcommand \eqvcs {équivalences\xspace}  

\newcommand \esid {essentiellement identique\xspace}  
\newcommand \esids {essentiellement identiques\xspace}  

\newcommand \Esid {\hyperref[defitdyesidentiques]{\esid}}  
\newcommand \Esids {\hyperref[defitdyesidentiques]{\esids}}  

\newcommand \eseq {essentiellement \eqve}  
\newcommand \eseqs {essentiellement \eqves}  

\newcommand \Eseq {\hyperref[defitheseq]{\eseq}}  
\newcommand \Eseqs {\hyperref[defitheseq]{\eseqs}}

\newcommand \fab {\fcn bornée\xspace}
\newcommand \fabs {\fcns bornées\xspace}

\newcommand \fat {\fcn totale\xspace}
\newcommand \fats {\fcn totales\xspace}

\newcommand \fap {\fcn partielle\xspace}
\newcommand \faps {\fcns partielles\xspace}

\newcommand \fip {filtre premier\xspace}
\newcommand \fips {filtres premiers\xspace}

\newcommand \fipma {\fip maximal\xspace}
\newcommand \fipmas {\fips maximaux\xspace}

\newcommand \fcn {factorisation\xspace}
\newcommand \fcns {factorisations\xspace}

\newcommand \fdi {fortement discret\xspace}
\newcommand \fdis {fortement discrets\xspace}

\newcommand \fsa {fermé \sagq}
\newcommand \fsas {fermés \sagqs}

\newcommand \fsagc {fonction \sagc}
\newcommand \fsagcs {fonctions \sagcs}

\newcommand \fmt {formellement\xspace}

\newcommand \fit {fidèlement\xspace}
\newcommand \fpt {\fit plat\xspace}
\newcommand \fpte {\fit plate\xspace}
\newcommand \fpts {\fit plats\xspace}
\newcommand \fptes {\fit plates\xspace}

\newcommand \frl {fortement réticulé\xspace}
\newcommand \frle {fortement réticulée\xspace}
\newcommand \frls {fortement réticulés\xspace}

\newcommand \ftm {fortement\xspace}

\newcommand\gmt{géométrie\xspace}  
\newcommand\gmts{géométries\xspace}  

\newcommand\gaq{\gmt \agq}  

\newcommand\gmq{géométrique\xspace}  
\newcommand\gmqs{géométriques\xspace}  

\newcommand\gmqt{géométriquement\xspace}  

\newcommand\gne{généralisé\xspace}  
\newcommand\gnee{généralisée\xspace}  
\newcommand\gnes{généralisés\xspace}  
\newcommand\gnees{généralisées\xspace}  

\newcommand\gnl{général\xspace}  
\newcommand\gnle{générale\xspace}  
\newcommand\gnls{généraux\xspace}  
\newcommand\gnles{générales\xspace}  

\newcommand\gnlt{généralement\xspace}  

\newcommand\gnn{généralisation\xspace}  
\newcommand\gnns{généralisations\xspace}  

\newcommand\gnq {générique\xspace}  
\newcommand\gnqs {génériques\xspace}  

\newcommand\gnr{généraliser\xspace}  

\newcommand \gns{généralise\xspace}

\newcommand \gnt{généralité\xspace}
\newcommand \gnts{généralités\xspace}

\newcommand \grl{groupe \rtl}
\newcommand \grls{groupes \rtls}

\newcommand \gRl {\hyperref[theorieGrl]{\grl}}
\newcommand \gRls {\hyperref[theorieGrl]{\grls}}

\newcommand\gtr{générateur\xspace}  
\newcommand\gtrs{générateurs\xspace}  


\newcommand \hen {hensélien\xspace}
\newcommand \hens {henséliens\xspace}

\newcommand \homo {homomorphisme\xspace}
\newcommand \homos {homomorphismes\xspace}

\newcommand \hmg {homogène\xspace}
\newcommand \hmgs {homogènes\xspace}

\newcommand \icftr {intervalle compact réticulé \ftm réel\xspace}
\newcommand \icftrs {intervalles compacts réticulés \ftm réels\xspace}

\newcommand \icl {intégralement clos\xspace}
\newcommand \icle {intégralement close\xspace}

\newcommand \icsr {intervalle compact \stm réticulé\xspace}
\newcommand \icsrs {intervalles compacts \stm réticulés\xspace}

\newcommand \icrc {intervalle compact réel clos\xspace}
\newcommand \icrcs {intervalles compact réels clos\xspace}

\newcommand \id {idéal\xspace}
\newcommand \ids {idéaux\xspace}

\newcommand \ida {\idt \agq}
\newcommand \idas {\idts \agqs}

\newcommand \idc  {\idt de Cramer\xspace}
\newcommand \idcs {\idts de Cramer\xspace}

\newcommand \idd {idéal déterminantiel\xspace}
\newcommand \idds {idéaux déterminantiels\xspace}

\newcommand \idema {idéal maximal\xspace}
\newcommand \idemas {idéaux maximaux\xspace}

\newcommand \idep {idéal premier\xspace}
\newcommand \ideps {idéaux premiers\xspace}

\newcommand \idemi {\idep minimal\xspace}
\newcommand \idemis {\ideps minimaux\xspace}

\newcommand \idf {idéal de Fitting\xspace}
\newcommand \idfs {idéaux de Fitting\xspace}

\newcommand \idif {idéal \dvlg fini\xspace}
\newcommand \idifs {idéaux \dvlgs finis\xspace}

\newcommand \idli {idéal \dvli\xspace} 
\newcommand \idlis {idéaux \dvlis\xspace} 

\newcommand \idm {idempotent\xspace}
\newcommand \idms {idempotents\xspace}
\newcommand \idme {idempotente\xspace}
\newcommand \idmes {idempotentes\xspace}

\newcommand \idp {idéal principal\xspace}
\newcommand \idps {idéaux principaux\xspace}

\newcommand \idt {identité\xspace}
\newcommand \idts {identités\xspace}

\newcommand \idtr {indéterminée\xspace}
\newcommand \idtrs {indéterminées\xspace}

\newcommand \ifr {idéal fractionnaire\xspace}
\newcommand \ifrs {idéaux fractionnaires\xspace}

\newcommand \imd {immédiat\xspace}
\newcommand \imde {immédiate\xspace}
\newcommand \imds {immédiats\xspace}
\newcommand \imdes {immédiates\xspace}

\newcommand \imdt {immédiatement\xspace}

\newcommand \indtr {inf-demi-treillis\xspace} 

\newcommand \inteq {intuitivement \eqve}
\newcommand \inteqs {intuitivement \eqves}

\newcommand \Inteq {\hyperref[defextintequiv]{\inteq}}
\newcommand \Inteqs {\hyperref[defextintequiv]{\inteqs}}

\newcommand \ing {inverse \gne}
\newcommand \ings {inverses \gnes}

\newcommand \iMP {inverse de Moore-Penrose\xspace}
\newcommand \iMPs {inverses de Moore-Penrose\xspace}

\newcommand \ipp {\idep potentiel\xspace}
\newcommand \ipps {\ideps potentiels\xspace}

\newcommand \ird {irréductible\xspace}
\newcommand \irds {irréductibles\xspace}

\newcommand \iso {isomorphisme\xspace}
\newcommand \isos {isomorphismes\xspace}

\newcommand \isoc {isomorphe\xspace}
\newcommand \isocs {isomorphes\xspace}

\newcommand \itf {idéal \tf}
\newcommand \itfs {idéaux \tf}

\newcommand \itid {intuitivement identique\xspace}
\newcommand \itids {intuitivement identiques\xspace}

\newcommand \iv {inversible\xspace}
\newcommand \ivs {inversibles\xspace}

\newcommand \ivdg {inverse divisoriel\xspace} 
\newcommand \ivdgs {inverses divisoriels\xspace} 

\newcommand \ivde {inverse divisorielle\xspace} 
\newcommand \ivdes {inverses divisorielles\xspace} 

\newcommand \ivda {inverse divisoriel\xspace} 
\newcommand \ivdas {inverses divisoriels\xspace} 


\newcommand \lgb {local-global\xspace}
\newcommand \lgbe {locale-globale\xspace}
\newcommand \lgbs {local-globals\xspace}

\newcommand \LHe {Lemme de Hensel\xspace}
\newcommand \lHe {lemme de Hensel\xspace}
\newcommand \lHes {lemmes de Hensel\xspace}
\newcommand \LHm {\LHe multivarié\xspace}
\newcommand \lHm {\lHe multivarié\xspace}
\newcommand \lHms {\lHes multivariés\xspace}

\newcommand \lin {linéaire\xspace}
\newcommand \lins {linéaires\xspace}

\newcommand \lint {linéairement\xspace}

\newcommand \lmo {\lot monogène\xspace}
\newcommand \lmos {\lot monogènes\xspace}

\newcommand \lnl {\lot \nl}
\newcommand \lnls {\lot \nls}

\newcommand \lot {localement\xspace}

\newcommand \lon {localisation\xspace}
\newcommand \lons {localisations\xspace}

\newcommand \lop {\lot principal\xspace}
\newcommand \lops {\lot principaux\xspace}

\newcommand \lsdz {\lot \sdz}

\newcommand \mdi {module des \diles}

\newcommand \mlm {module \lmo}
\newcommand \mlms {modules \lmos}

\newcommand \mlmo {matrice de localisation
monogène\xspace}
\newcommand \mlmos {matrices de localisation
monogène\xspace}

\newcommand \mlp {matrice de localisation
principale\xspace}
\newcommand \mlps {matrices de localisation
principale\xspace}

\newcommand \mo {mono\"{\i}de\xspace}
\newcommand \mos {mono\"{\i}des\xspace}

\newcommand \moco {\mos \com}

\newcommand \molo {morphisme de \lon\xspace}
\newcommand \molos {morphismes de \lon\xspace}

\newcommand \mom {monôme\xspace}
\newcommand \moms {monômes\xspace}

\newcommand \moquo {morphisme de passage au quotient\xspace}
\newcommand \moquos {morphismes de passage au quotient\xspace}

\newcommand \mpf {module \pf}
\newcommand \mpfs {modules \pf}

\newcommand \mpl {module plat\xspace}
\newcommand \mpls {modules plats\xspace}

\newcommand \mpn {matrice de \pn}
\newcommand \mpns {matrices de \pn}

\newcommand \mprn {matrice de \prn}
\newcommand \mprns {matrices de \prn}

\newcommand \mptf {module \ptf}
\newcommand \mptfs {modules \ptfs}

\newcommand \mrc {module \prc}
\newcommand \mrcs {modules \prcs}

\newcommand \mtf {module \tf}
\newcommand \mtfs {modules \tf}


\newcommand \ncr{nécessaire\xspace}       
\newcommand \ncrs{nécessaires\xspace}       

\newcommand \ncrt{nécessairement\xspace}       

\newcommand \ndz {\reg}
\newcommand \reg {régulier\xspace}
\newcommand \regs {réguliers\xspace}
\newcommand \ndzs {\regs}

\newcommand \nl {simple\xspace}
\newcommand \nls {simples\xspace}

\newcommand \noco {\noe \coh}
\newcommand \nocos {\noes \cohs}

\newcommand \Noe {Noether\xspace}

\newcommand \noe {noethérien\xspace}
\newcommand \noes {noethériens\xspace}
\newcommand \noee {noethérienne\xspace}
\newcommand \noees {noethériennes\xspace}

\newcommand \noet {noethérianité\xspace}

\newcommand \nst {Nullstellensatz\xspace}
\newcommand \nsts {Nullstellens\"atze\xspace}

\newcommand \op{opération\xspace}  
\newcommand \ops{opérations\xspace}
\newcommand \opari{\op\ari}  
\newcommand \oparis{\ops\aris}  
\newcommand \oparisv{\ops\arisv}  

\newcommand \oqc {ouvert \qc}
\newcommand \oqcs {ouverts \qcs}

\newcommand \ort{orthogonal\xspace}  
\newcommand \orte{orthogonale\xspace}  
\newcommand \orts{orthogonaux\xspace}  
\newcommand \ortes{orthogonales\xspace}  


\newcommand \pa {couple saturé\xspace}
\newcommand \pas {couples saturés\xspace}
 
\newcommand \paral{parallèle\xspace}  
\newcommand \parals{paallèles\xspace}  

\newcommand \paralm{parallèlement\xspace}

\newcommand \pb{problème\xspace}  
\newcommand \pbs{problèmes\xspace}  

\newcommand \peq {purement équationnelle\xspace}
\newcommand \peqs {purement équationnelles\xspace}

\newcommand \pf {de \pn finie\xspace}

\newcommand \pgn {polygone de Newton\xspace}
\newcommand \pgns {polygones de Newton\xspace}

\newcommand \plc {\rdt \zed}
\newcommand \plcs {\rdt \zeds}

\newcommand \Plg {Principe \lgb}
\newcommand \plg {principe \lgb}
\newcommand \plgs {principes \lgbs}

\newcommand \plga {\plg abstrait\xspace}
\newcommand \plgas {\plgs abstraits\xspace}

\newcommand \Plgc {\Plg concret\xspace}
\newcommand \plgc {\plg concret\xspace}
\newcommand \plgcs {\plgs concrets\xspace}

\newcommand \pn {présentation\xspace}
\newcommand \pns {présentations\xspace}

\newcommand \pog {\pol \hmg\xspace}
\newcommand \pogs {\pols \hmgs\xspace}

\newcommand \Pol {Polynôme\xspace}
\newcommand \Pols {Polynômes\xspace}

\newcommand \pol {polynôme\xspace}
\newcommand \pols {polynômes\xspace}

\newcommand \polH {\pol de Hensel\xspace}
\newcommand \polHs {\pols de Hensel\xspace}

\newcommand \poll{polynomial\xspace}  
\newcommand \polls{polynomiaux\xspace}  
\newcommand \polle{polynomiale\xspace}  
\newcommand \polles{polynomiales\xspace}  

\newcommand \pollt{polynomialement\xspace}  

\newcommand \polfon {\pol fondamental\xspace}
\newcommand \polmu {\pol rang\xspace}
\newcommand \polmus {\pols rang\xspace}
\newcommand \polcar {\pol caractéristique\xspace}
\newcommand \polcars {\pols caractéristiques\xspace}
\newcommand \polmin {\pol minimal\xspace}
\newcommand \polmins {\pols minimaux\xspace}

\newcommand \polu {\pol unitaire\xspace}
\newcommand \polus {\pols unitaires\xspace}

\newcommand \prc {\pro de rang constant\xspace}
\newcommand \prcs {\pros de rang constant\xspace}

\newcommand \prcc {principe de \rcc}
\newcommand \prca {principe de \rca}
\newcommand \prce {principe de \rce}

\newcommand \prmt {précisément\xspace}
\newcommand \Prmt {Précisément\xspace}

\newcommand \prn {projection\xspace}
\newcommand \prns {projections\xspace}

\newcommand \pro {projectif\xspace}
\newcommand \pros {projectifs\xspace}

\newcommand \prr {projecteur\xspace}
\newcommand \prrs {projecteurs\xspace}

\newcommand \Prt {Propriété\xspace}
\newcommand \Prts {Propriétés\xspace}
\newcommand \prt {propriété\xspace}
\newcommand \prts {propriétés\xspace}

\newcommand \ptf {\pro \tf}
\newcommand \ptfs {\pros \tf}

\newcommand \qc {quasi-compact\xspace}
\newcommand \qcs {quasi-compacts\xspace}

\newcommand \qi {quasi intègre\xspace}
\newcommand \qis {quasi intègres\xspace}

\newcommand \qnl {quasi-\nl}
\newcommand \qnls {quasi-\nls}

\newcommand \ralg {règle \agq}
\newcommand \ralgs {règles \agqs}

\newcommand \raJ {radical de Jacobson\xspace}

\newcommand \rav {racine virtuelle\xspace}
\newcommand \ravs {racines virtuelles\xspace}

\newcommand \rcc {\rcm concret\xspace}
\newcommand \rca {\rcm abstrait\xspace}
\newcommand \rce {\rcc des égalités\xspace}

\newcommand \rcm {recollement\xspace}
\newcommand \rcms {recollements\xspace}

\newcommand \rcv {recouvrement\xspace} 
\newcommand \rcvs {recouvrements\xspace}

\newcommand \rde {relation de dépendance\xspace}
\newcommand \rdes {relations de dépendance\xspace}

\newcommand \rdi {\rde intégrale\xspace}
\newcommand \rdis {\rdes intégrales\xspace}

\newcommand \rdl {\rde \lin}
\newcommand \rdls {\rdes \lins}

\newcommand \rdt {résiduellement\xspace}

\newcommand \rdy {règle dynamique\xspace}
\newcommand \rdys {règles dynamiques\xspace}

\def \red {règle directe\xspace}
\newcommand \reds {règles directes\xspace}

\newcommand \rex {\hyperref[defexistsimple]{règle existentielle simple\xspace}}
\newcommand \rexs {\hyperref[defexistsimple]{règles existentielles simples\xspace}}

\newcommand \rexri {\hyperref[defitdyexrig]{règle existentielle rigide\xspace}}
\newcommand \rexris {\hyperref[defitdyexrig]{règles existentielles rigides\xspace}}

\newcommand \rsim {règle de simplification\xspace}
\newcommand \rsims {règles de simplification\xspace}

\newcommand \rtl {réticulé\xspace}
\newcommand \rtls {réticulés\xspace}

\newcommand \rmq {\rcm de quotients\xspace} 
\newcommand \rvq {\rcv par quotients\xspace} 
\newcommand \rmqs {\rcms de quotients\xspace} %
\newcommand \rvqs {\rcvs par quotients\xspace} %

\newcommand \rpf {réduite-de-présentation-finie\xspace}
\newcommand \rpfs {réduites-de-présentation-finie\xspace}


\newcommand \sad {\salg dynamique\xspace}
\newcommand \sads {\salgs dynamiques\xspace}

\newcommand \sagq {semi\agq}
\newcommand \sagqs {semi\agqs}

\newcommand \sagc {\sagq continue\xspace}
\newcommand \sagcs {\sagqs continues\xspace}

\newcommand \salg {structure \agq}
\newcommand \salgs {structures \agqs}

\newcommand \scentrel {relation semi-implicative\xspace}
\newcommand \scentrels {relations semi-implicatives\xspace}

\newcommand \scf {schéma finitaire\xspace}
\newcommand \scfs {schémas finitaires\xspace}

\newcommand \scl {schéma \elr}
\newcommand \scls {schémas \elrs}

\newcommand \sdo {\sdr \orte}
\newcommand \sdos {\sdrs \ortes}

\newcommand \sdr {somme directe\xspace}
\newcommand \sdrs {sommes directes\xspace}

\newcommand \sdz {sans \dvz}

\newcommand \sfio {système fondamental d'\idms orthogonaux\xspace}
\newcommand \sfios {systèmes fondamentaux d'\idms orthogonaux\xspace}

\newcommand \sgr {\sys \gtr}
\newcommand \sgrs {\syss \gtrs}

\newcommand \slgb {strictement \lgb}
\newcommand \slgbs {strictement \lgbs}

\newcommand \sli {\sys \lin}
\newcommand \slis {\syss \lins}

\newcommand \sml {semi-local\xspace}
\newcommand \smls {semi-locaux\xspace}

\newcommand \smq {symétrique\xspace}
\newcommand \smqs {symétriques\xspace}

\newcommand \spb {séparable\xspace}  
\newcommand \spbs {séparables\xspace}

\newcommand \spe {spécification\xspace}
\newcommand \spes {spécifications\xspace}

\newcommand \spi {\spe incomplète\xspace}
\newcommand \spis {\spes incomplètes\xspace}

\newcommand \spl {séparable\xspace}  
\newcommand \spls {séparables\xspace}

\newcommand \splc {\splt clos\xspace}
\newcommand \splce {\splt close\xspace}
\newcommand \splces {\splt closes\xspace}

\newcommand \splt {séparablement\xspace}  

\newcommand \spo {semipolynôme\xspace}
\newcommand \spos {semipolynômes\xspace}

\newcommand \spt{séparabilité\xspace}

\newcommand \srg {suite régulière\xspace}
\newcommand \srgs {suites régulières\xspace}

\newcommand \ste {strictement étale\xspace}
\newcommand \stes {strictement étales\xspace}

\newcommand \stf {strictement fini\xspace}
\newcommand \stfs {strictement finis\xspace}
\newcommand \stfe {strictement finie\xspace}
\newcommand \stfes {strictement finies\xspace}

\newcommand \stl {stablement libre\xspace}
\newcommand \stls {stablement libres\xspace}

\newcommand \stm {strictement\xspace}

\newcommand \str {\stm réticulé\xspace}
\newcommand \stre {\stm réticulée\xspace}
\newcommand \strs {\stm réticulés\xspace}
\newcommand \stres {\stm réticulées\xspace}

\newcommand \sul {supplémentaire\xspace}
\newcommand \suls {supplémentaires\xspace}

\newcommand \Sut {Support\xspace}
\newcommand \Suts {Supports\xspace}
\newcommand \sut {support\xspace}

\newcommand \syc {\sys de coordonnées\xspace}
\newcommand \sycs {\syss de coordonnées\xspace}

\newcommand \syp {\sys \poll}
\newcommand \Syp {\Sys \poll}
\newcommand \syps {\syss \polls}

\newcommand \sype {\syp étale\xspace}
\newcommand \sypes {\syps étales\xspace}

\newcommand \sys {système\xspace}
\newcommand \Sys {Système\xspace}
\newcommand \syss {systèmes\xspace}

\newcommand \sysN {\sys de Newton\xspace}
\newcommand \SysN {\Sys de Newton\xspace}
\newcommand \sysNs {\syss de Newton\xspace}

\newcommand \sysNe {\sysN étale\xspace}
\newcommand \sysNes {\sysNs étales\xspace}

\newcommand \syzy {syzygie\xspace}
\newcommand \syzys {syzygies\xspace}

\newcommand \talg {théorie \agq}
\newcommand \talgs {théories \agqs}

\newcommand \tco {théorie cohérente\xspace}
\newcommand \tcos {théories cohérentes\xspace}

\newcommand \tdij {théorie disjonctive\xspace}
\newcommand \tdijs {théories disjonctives\xspace}

\newcommand \tdy {théorie dynamique\xspace}
\newcommand \tdys {théories dynamiques\xspace}

\newcommand \tel {\hyperref[defexistsimple]{théorie existentielle\xspace}}
\newcommand \tels {\hyperref[defexistsimple]{théories existentielles\xspace}}

\newcommand \telri {\hyperref[defitdyexrig]{théorie existentielle rigide\xspace}}
\newcommand \telris {\hyperref[defitdyexrig]{théories existentielles rigides\xspace}}

\newcommand \tf {de type fini\xspace}

\newcommand \tfo {théorie formelle\xspace}
\newcommand \tfos {théorie formelles\xspace}

\newcommand \tgm {théorie \gmq}
\newcommand \tgms {théories \gmqs}

\newcommand \Tho {Théorème\xspace}
\newcommand \Thos {Théorèmes\xspace}
\newcommand \tho {théorème\xspace}
\newcommand \thos {théorèmes\xspace}

\newcommand \thoc {théorème$\mathbf{^*}$~}

\newcommand \tpe {théorie \peq}
\newcommand \tpes {théories \peqs}

\newcommand \trdi {treillis distributif\xspace}
\newcommand \trdis {treillis distributifs\xspace}

\newcommand \trel {transformation \elr}
\newcommand \trels {transformations \elrs}


\newcommand \ultm {ultramétrique\xspace}
\newcommand \ultms {ultramétriques\xspace}

\newcommand \ultd {\ultm discret\xspace}
\newcommand \ultds {\ultms discrets\xspace}

\newcommand \umd {unimodulaire\xspace}
\newcommand \umds {unimodulaires\xspace}

\newcommand \unt {unitaire\xspace}
\newcommand \unts {unitaires\xspace}

\newcommand \uvl {universel\xspace}
\newcommand \uvle {universelle\xspace}
\newcommand \uvls {universels\xspace}
\newcommand \uvles {universelles\xspace}


\newcommand \vala {valeur absolue\xspace}
\newcommand \valas {valeurs absolues\xspace}

\newcommand \valn {valuation\xspace}
\newcommand \valns {valuations\xspace}

\newcommand \valu {\vala \ultm}
\newcommand \valus {\valas \ultms}

\newcommand \vfn {vérification\xspace}
\newcommand \vfns {vérifications\xspace}

\newcommand \vmd {vecteur \umd}
\newcommand \vmds {vecteurs \umds}

\newcommand \vst {Valuativstellensatz\xspace}
\newcommand \vsts {Valuativstellensätze\xspace}

\newcommand \vstf {\vst formel\xspace}


\newcommand \zeH {zéro de Hensel\xspace}
\newcommand \zeHs {zéros de Hensel\xspace}

\newcommand \zed {zéro-dimensionnel\xspace}
\newcommand \zede {zéro-dimensionnelle\xspace}
\newcommand \zeds {zéro-dimensionnels\xspace}
\newcommand \zedes {zéro-dimensionnelles\xspace}

\newcommand \zedr {\zed réduit\xspace}
\newcommand \zedre {\zede réduite\xspace}
\newcommand \zedrs {\zeds réduits\xspace}

\newcommand \zmt {\tho de Zariski-Grothendieck\xspace}


\newcommand \cof {constructif\xspace}
\newcommand \cofs {constructifs\xspace}

\newcommand \cov {constructive\xspace}
\newcommand \covs {constructives\xspace}

\newcommand \coma {\maths\covs}
\newcommand \clama {\maths classiques\xspace}

\renewcommand \cot {constructivement\xspace}

\newcommand \matn {mathématicien\xspace}
\newcommand \matne {mathématicienne\xspace}
\newcommand \matns {mathématiciens\xspace}
\newcommand \matnes {mathématiciennes\xspace}

\newcommand \maths {mathématiques\xspace}
\newcommand \mathe {mathématique\xspace}

\newcommand \prco {démonstration \cov}
\newcommand \prcos {démonstrations \covs}

\newcommand {\junk}[1]{}

\newcommand\rouge[1]{\textcolor{red}{#1}}
\newcommand\bleu[1]{\textcolor{blue}{#1}}
\newcommand\violet[1]{\textcolor{magenta}{#1}}

\newcommand{\Cadre}[2]{%

\medskip%
\newskip\oldleftskip
\newskip\oldrightskip
\oldleftskip=\leftskip%
\oldrightskip=\rightskip%
\leftskip=-\tabcolsep%
\rightskip=-\tabcolsep%
\begin{center}\fbox{%
\begin{tabular}%
{p{#1\textwidth}}
\setlength{\parindent}{5mm}%
\vspace{-1.5mm}#2\vspace{1mm}%
\end{tabular}}\end{center}\par\medskip%
\leftskip=\oldleftskip%
\rightskip=\oldrightskip%
\setlength{\parindent}{6mm}}

\newcommand\boite[2]{\begin{minipage}[c]{#1cm}
     \centering {#2} \end{minipage}}
\newcommand\Boite[3]{\parbox[t][#1cm][c]{#2cm}{\boite{#2}{#3}}}

\newcommand{\Encadre}[1]{\Cadre{.8}{#1}}

\newcommand{\Cencadre}[1]{\Encadre{\vspace{-3mm}
\begin{center}
#1 \end{center}\vspace{-8mm}}}

\newcommand{\cen}{\centerline}
\newcommand \Grandcadre[1]{%
\begin{center}
\begin{tabular}{|c|}
\hline
~\\[-3mm]
#1\\[-3mm]
~\\
\hline
\end{tabular}
\end{center}}

\newcommand\dsp{\displaystyle}
\newcommand\ndsp{\textstyle}

\newcommand{\eop}{\hfill \mbox{$\Box$}}

\newcommand \noi {\noindent}
\renewcommand \ss {ou la la}
\newcommand \sni {\smallskip \noi}
\newcommand \snii {\noi}
\newcommand \ms {\medskip}
\newcommand \mni {\ms\noi}
\newcommand \bs {\bigskip}
\newcommand \bni {\bs\noi}
\newcommand \hs {\qquad}
\newcommand \alb {\allowbreak}
\newcommand \ce {\centerline}


\renewcommand \le{\leqslant}
\renewcommand \leq{\leqslant}
\renewcommand \preceq{\preccurlyeq}
\renewcommand \ge{\geqslant}
\renewcommand \geq{\geqslant}
\renewcommand \succeq{\succcurlyeq}
\newcommand   \nneq {\mathrel{\#}}
\newcommand   \ineq {$_{\,\mathrel{\#}}$}

\newcommand\eti{^\times}
\newcommand \epr{^\perp}
\newcommand \etl{^*}
\newcommand \sta{^\star}
\newcommand \bu {{$\bullet$}}
\newcommand \bl {^\bullet}
\newcommand{\bul}{^{\bullet}}
\newcommand \eci {^\circ}
\newcommand \uci{\mathring}
\newcommand \ep[1]{^{(#1)}}
\newcommand \esh{^\sharp}
\newcommand \efl{^\flat}
\newcommand \eto{$^*$ }
\newcommand \etoz{$^*$}
\newcommand \ist{_\star}

\newcommand \Ast {\gA^{\!\star}}
\newcommand \Bst {\gB^{\star}}
\newcommand \Bo{\BB\mathrm{o}}
\newcommand \Ati {\gA^{\!\times}}
\newcommand \Bti {\gB^{\times}}
\newcommand \Vti {\gV^{\times}}
\newcommand \Atl {\gA^{\!*}}
\newcommand \Btl {\gB^{*}}
\newcommand \Ktl {\gK^{*}}
\newcommand \Vtl {\gV^{*}}
\newcommand{\KAt}{\gK\etl\!\sur{\Ati}}
\newcommand{\AAt}{\Atl\!\sur{\Ati}}

\newcommand \Ared {\gA_{\mathrm{red}}}
\newcommand \Bred {\gB_{\mathrm{red}}}

\newcommand \iBA {_{\gB/\!\gA}}
\newcommand \iWV {_{\gW\!/\gV}}
\newcommand \iCA {_{\gC/\!\gA}}
\newcommand \iBk {_{\gB/\gk}}
\newcommand \iBK {_{\gB/\gK}}
\newcommand \iAk {_{\gA/\gk}}
\newcommand \iAK {_{\gA/\gK}}
\newcommand \iCk {_{\gC/\gk}}

\newcommand \divi {\mid}
\newcommand \nedivi {\not\kern 2.5pt\mid}

\newcommand\equidef{\buildrel{{\rm déf}}\over{\;\Longleftrightarrow\;}}
\newcommand\eqdef{\buildrel{\rm déf}\over {\;=\;}}
\newcommand\eqdefi{\buildrel{\rm déf}\over {\;=\;}}


\newcommand \fraC[2] {{{#1}\over {#2}}}
\newcommand \formule[1]{{\left\{ {\arraycolsep2pt\begin{array}{lll} #1 \end{array}}\right.}}
\newcommand \formul[1]{{\left\{ {\arraycolsep2pt\begin{array}{rcl} #1 \end{array}}\right.}}

\newcommand\mapright[1]{\smash{\mathop{\longrightarrow}\limits^{#1}}}
\newcommand\maprightto[1]{\smash{\mathop{\longmapsto}\limits^{#1}}}
\newcommand\mapdown[1]{\downarrow\rlap{$\vcenter{\hbox{$\scriptstyle
#1$}}$}}
\newcommand{\pref}[1]{\textup{\hbox{\normalfont(\ref{#1})}}}

\newcommand \abs[1] {\left|{#1}\right|}
\newcommand \abS[1] {\big|{#1}\big|}
\newcommand \aqo[2] {#1\sur{\gen{#2}}\!}
\newcommand \aQo[2] {#1/{\gEn{#2}}\!}
\newcommand \Aqo[2] {#1\sur{\big\langle{#2}\big\rangle}\!}
\newcommand \Al[1] {\Vi^{\!#1}}
\newcommand \ci[1] {{{#1}^\circ}}
\newcommand \crac[2] {\cro {\frac{#1}{#2}}}
\newcommand \cro[1] {\left[#1\right]}
\newcommand \eqdf[1] {\buildrel{#1}\over =}
\newcommand \equivdf[1] {\buildrel{#1}\over \longleftrightarrow}
\newcommand \frt[1] {\!\left|_{#1}\right.\!}
\newcommand \impdef[1] {\buildrel{#1}\over \Longrightarrow}
\newcommand \norme[1]{\left\lVert #1 \right\rVert}
\newcommand \Norme[1]{\big\lVert #1 \big\rVert}
\newcommand \tra[1] {{\,^{\rm t}\!#1}}
\newcommand \gen[1] {\left\langle{#1}\right\rangle}
\newcommand \gEn[1] {\langle{#1}\rangle}
\newcommand \geN[1] {\big\langle{#1}\big\rangle}
\newcommand \sing[1] {\left\{{#1}\right\}}
\newcommand \so[1] {\left\{\,{#1}\, \right\}}
\newcommand \soo[1] {\{\,{#1}\,\}}
\newcommand \sO[1]{\big\{{#1}\big\}}
\newcommand \sotq[2]{\so{#1\mathrel{;}#2}}
\newcommand \sootq[2]{\soo{#1\mathrel{;}#2}}
\newcommand \sotQ[2]{\sO{#1\mathrel{;}#2}}
\newcommand \sur[1] {\!\left/#1\right.}
\newcommand \und[1] {\underline{#1}}

\newcommand \Sqr {\mathrm{Sqr}}

\newcommand \idg[1] {|\,#1\,|}
\newcommand \idG[1] {\big|\,#1\,\big|}

\newcommand \norm[1] {\Vert\,#1\,\Vert}

\newcommand \dex[1] {[\,#1\,]}
\newcommand \deX[1] {\big[\,#1\,\big]}

\newcommand \lst[1] {[\,#1\,]}
\newcommand \lsT[1] {\big[\,#1\,\big]}

\newcommand{\mt}{\mapsto}

\newcommand{\llongrightarrow}{\relbar\joinrel\mkern-1mu\longrightarrow}
\newcommand{\lllongrightarrow}{\relbar\joinrel\mkern-1mu\llongrightarrow}
\newcommand{\llllongrightarrow}{\relbar\joinrel\mkern-1mu\lllongrightarrow}
\newcommand\simarrow{\vers{_\sim}}
\newcommand\isosim{\buildrel{_\sim}\over \longleftrightarrow }
\newcommand\vers[1]{\buildrel{#1}\over \longrightarrow }
\newcommand\vvers[1]{\buildrel{#1}\over \llongrightarrow }
\newcommand\vvvers[1]{\buildrel{#1}\over \lllongrightarrow }
\newcommand \lora {\longrightarrow}
\newcommand \llra {\llongrightarrow}
\newcommand \lllra {\lllongrightarrow}

\renewcommand \leq{\leqslant}
\renewcommand \geq{\geqslant}

\newcommand \som {\sum\nolimits}
\newcommand \Ex {{\exists}}
\newcommand \Tt {{\forall}}
\newcommand \te {\otimes}
\newcommand \vep{{\varepsilon}}


\newcommand\lra[1]{\langle{#1}\rangle}
\newcommand\lrb[1] {\llbracket #1 \rrbracket}
\newcommand\lrbd {\lrb{1..d}}
\newcommand\lrbn {\lrb{1..n}}
\newcommand\lrbl {\lrb{1..\ell}}
\newcommand\lrbm {\lrb{1..m}}
\newcommand\lrbk {\lrb{1..k}}
\newcommand\lrbp {\lrb{1..p}}
\newcommand\lrbq {\lrb{1..q}}
\newcommand\lrbr {\lrb{1..r}}
\newcommand\lrbs {\lrb{1..s}}


\newcommand \vda {\,\vdash\,}

\newcommand \vdw {\,\vdash_w\,}

\newcommand\Aq[2]{#1_{[#2]}}
\newcommand\Aqj[2]{#1_{\{#2\}}}

\newcommand \dar[1] {\MA{\downarrow \!#1}}
\newcommand \uar[1] {\MA{\uparrow \!#1}}
\newcommand \clps[1] {{\downarrow #1\,\downarrow}}
\newcommand \lrce[1] {\lceil#1\rceil}
\newcommand \Lrce[1] {\left\lceil#1\right\rceil}
\newcommand \lrfl[1] {\lfloor#1\rfloor}

\newcommand\tsbf[1]{\textsf{\textbf{\textup{#1}}}}
\newcommand\lab[1]{\item[\tsbf{#1}]}
\newcommand\Lab[1]{\rdb\item[\tsbf{#1}]\label{Ax#1}}
\newcommand\Tsbf[1]{\hyperref[Ax#1]{\tsbf{#1}}}

\newcommand\SA[1]{\rdb\sa{#1}\label{theorie#1}}
\newcommand\Sa[1]{\hyperref[theorie#1]{\sa{#1}}}
\newcommand\sa[1]{\hbox{\usefont{T1}{pzc}{m}{it}#1}\,}
\newcommand\sab[1]{\hbox{\usefont{T1}{pzc}{m}{it}#1}\,\,}
\newcommand\sA[1]{\hbox{\small\usefont{T1}{pzc}{m}{it}#1}\,}
\newcommand\sAb[1]{\hbox{\small\usefont{T1}{pzc}{m}{it}#1}\,\,}

\newcommand\Ae[1]{\gA^{\!#1}}

\newcommand \snic[1] {\sni\centerline{$#1$}

\ss}

\newcommand \snac[1]{\sni
{\small\centering$#1$\par}

\ss}

\newcommand \eoe {\hbox{}\nobreak\hfill
\vrule width .5em height .5em depth 0mm \par \smallskip}

\newcommand \bal[1] {^\rK_{#1}}
\newcommand \ul[1] {_\rK^{#1}}

\newcommand\env[2] {{{#2}_{#1}^{\mathrm{e}}}} 
\newcommand\Om[2]  {\Omega_{{#2}/{#1}}}
\newcommand\Der[3] {{\rm Der}_{{#1}}({#2},{#3})}

\newcommand{\DBxk}{{\Der \gk\gB\xi}}%
\newcommand{\DAxk}{{\Der \gk\gA\xi}}%
\newcommand{\DkXxk}{{\Der \gk\kuX\xi}}%

\newcommand \isA[1] {_{#1/\!\gA}}
\newcommand \OmA[1]{\Omega\isA{#1}}

\newcommand \ov[1] {\overline{#1}}

\newcommand \wh[1] {\widehat{#1} }
\newcommand \wi[1] {\widetilde{#1} }

\newcommand\dessus[2]{{\textstyle {#1} \atop \textstyle {#2}}}

\newcommand\carray[2]{{\left[\begin{array}{#1} #2 \end{array}\right]}}
\newcommand\cmatrix[1]{\left[\matrix{#1}\right]}
\newcommand\clmatrix[1]{{\left[\begin{array}{lllllll} #1 \end{array}\right]}}
\newcommand\dmatrix[1]{\abs{\matrix{#1}}}
\newcommand\Cmatrix[2]{\setlength{\arraycolsep}{#1}\left[\matrix{#2}\right]}
\newcommand\Dmatrix[2]{\setlength{\arraycolsep}{#1}\left|\matrix{#2}\right|}


\makeatletter
\newif\if@borderstar
\def\bordercmatrix{\@ifnextchar*{%
  \@borderstartrue\@bordercmatrix@i}{\@borderstarfalse\@bordercmatrix@i*}%
}
\def\@bordercmatrix@i*{\@ifnextchar[{%
  \@bordercmatrix@ii}{\@bordercmatrix@ii[()]}
}
\def\@bordercmatrix@ii[#1]#2{%
  \begingroup
    \m@th\@tempdima.875em\setbox\z@\vbox{%
      \def\cr{\crcr\noalign{\kern 2\p@\global\let\cr\endline}}%
      \ialign {$##$\hfil\kern.2em\kern\@tempdima&\thinspace%
      \hfil$##$\hfil&&\quad\hfil$##$\hfil\crcr\omit\strut%
      \hfil\crcr\noalign{\kern-\baselineskip}#2\crcr\omit%
      \strut\cr}}%
    \setbox\tw@\vbox{\unvcopy\z@\global\setbox\@ne\lastbox}%
    \setbox\tw@\hbox{\unhbox\@ne\unskip\global\setbox\@ne\lastbox}%
    \setbox\tw@\hbox{%
      $\kern\wd\@ne\kern-\@tempdima\left\@firstoftwo#1%
        \if@borderstar\kern.2em\else\kern -\wd\@ne\fi%
      \global\setbox\@ne\vbox{\box\@ne\if@borderstar\else\kern.2em\fi}%
      \vcenter{\if@borderstar\else\kern-\ht\@ne\fi%
        \unvbox\z@\kern-\if@borderstar2\fi\baselineskip}%
\if@borderstar\kern-2\@tempdima\kern.4em\else\,\fi\right\@secondoftwo#1 $%
    }\null\;\vbox{\kern\ht\@ne\box\tw@}%
  \endgroup
}
\makeatother

\newcommand \bloc[4]{\left[
\begin{array}{cc}
#1 & #2   \\
#3 & #4
\end{array}
\right]}

\newcommand \cm{em}

\newcommand{\blocs}[8]{%
{\setlength{\unitlength}{.0833\textwidth}
\tabcolsep0pt\renewcommand{\arraystretch}{0}%
\begin{tabular}{|c|c|}
\hline
\parbox[t][#3\cm][c]{#1\cm}{\begin{minipage}[c]{#1\cm}
\centering#5
\end{minipage}}&
\parbox[t][#3\cm][c]{#2\cm}{\begin{minipage}[c]{#2\cm}
\centering#6
\end{minipage}}\\
\hline
\parbox[t][#4\cm][c]{#1\cm}{\begin{minipage}[c]{#1\cm}
\centering#7
\end{minipage}}&
\parbox[t][#4\cm][c]{#2\cm}{\begin{minipage}[c]{#2\cm}
\centering#8
\end{minipage}}\\
\hline
\end{tabular}
}}

\newcommand \UneCol[1]{%
\sni\mbox{\hspace{.02\textwidth}%
\parbox[t]{.98\textwidth}{#1}%
}}

\newcommand \Unecol[1]{%
\sni\mbox{\hspace{.1\textwidth}%
\parbox[t]{.9\textwidth}{#1}%
}}

\newcommand \Deuxcol[4]{%
\sni\mbox{\parbox[t]{#1\textwidth}{#3}%
\hspace{.05\textwidth}%
\parbox[t]{#2\textwidth}{#4}}}

\newcommand \DeuxCol[2]{%
\Deuxcol{.475}{.475}{#1}{#2}}

\newcommand \DeuxCols[2]{%
\sni\mbox{\hspace{.02\textwidth}%
\parbox[t]{.475\textwidth}{#1}%
\hspace{.03\textwidth}%
\parbox[t]{.475\textwidth}{#2}}}

\newcommand \DeuxRegles[2]{%
\vspace{-1em}\DeuxCols
{\begin{enumerate}  #1
\end{enumerate}
}
{\begin{enumerate}  #2
\end{enumerate}
}
\vspace{-.3em}
}

\newcommand \UneRegle[2]{%
\vspace{-1em}\UneCol{
\begin{enumerate}
\lab{#1}{#2}
\end{enumerate}
}
\vspace{-.3em}
}

\newcommand \Regles[1]{%
\vspace{-1em}\UneCol{
\begin{enumerate}
{#1}
\end{enumerate}
}
\vspace{-.3em}
}

\newcommand \regles[1]{%
\vspace{-1em}\Unecol{
\begin{enumerate}
{#1}
\end{enumerate}
}
\vspace{-.3em}
}

\newcommand \itbu {\item[$\bullet$]}
\newcommand \labu {\lab{$\bullet$}}

\newcommand \Deuxbu[2]{%
\vspace{-.6em}\DeuxCols
{\begin{itemize}  #1
\end{itemize}}
{\begin{itemize}  #2
\end{itemize}}}

\newcommand\dcan[9]{
\xymatrix @C=1.2cm{
#1\,\ar[d]_{#4}\ar[r]^{#2}   & \,#3   \\
#6\,\ar[r]_{#7}    & #8\ar[u]_{#5}  
}
}

\newcommand \PNV[9]{
$$\quad
\vcenter{\xymatrix@C=2.5cm @R=1.7cm
{
#1 \ar[d]_{#2} \ar[dr]^{#3} \\
{#4} \ar@{-->}[r]_{{#5}\,!}   & {#6} \\
}}
\quad\quad
\vcenter{
\vspace{2mm}
\hbox{\small {#7}}
\hbox{~\\[2mm] ~}
\hbox{\small {#8}}
\hbox{~\\[2mm] ~}
\hbox{\small {#9}}
\hbox{~\\[2mm] ~ }
}
$$
}

\makeatletter
\def\revddots{\mathinner{\mkern1mu\raise\p@
\vbox{\kern7\p@\hbox{.}}\mkern2mu
\raise4\p@\hbox{.}\mkern2mu\raise7\p@\hbox{.}\mkern1mu}}
\makeatother

\newcommand \BB{\mathbb {B}}
\newcommand \CC{\mathbb {C}}
\newcommand \FF{\mathbb {F}}
\newcommand \II{\mathbb {I}}
\newcommand \KK{\mathbb {K}}
\newcommand \MM{\mathbb {M}}
\newcommand \NN{\mathbb {N}}
\newcommand \ZZ{\mathbb {Z}}
\newcommand \OO{\mathbb {O}}
\newcommand \PP{\mathbb {P}}
\newcommand \QQ{\mathbb {Q}}
\newcommand \RR{\mathbb {R}}

\newcommand \Fp{\FF_p}
\newcommand \Qp{\QQ_p}
\newcommand \Zp{\ZZ_p}
\newcommand \Cp{\CC_p}
\newcommand \Fpa{\FF_{p,\rm alg}}
\newcommand \Qpa{\QQ_{p,\rm alg}}
\newcommand \Zpa{\ZZ_{p,\rm alg}}
\newcommand \Cpa{\CC_{p,\rm alg}}
\newcommand \QpX{\Qp[X]}
\newcommand \FpX{\Fp[X]}
\newcommand \ZpX{\Zp[X]}
\newcommand \ZpXn{\Zp[\Xn]}
\newcommand \CpX{\Cp[X]}
\newcommand \QpaX{\Qpa[X]}
\newcommand \QpaY{\Qpa[Y]}
\newcommand \Qpax{\Qpa[x]}
\newcommand \Qpay{\Qpa[y]}
\newcommand \ZpaX{\Zpa[X]}
\newcommand \ZpaY{\Zpa[Y]}
\newcommand \Zpax{\Zpa[x]}
\newcommand \Zpay{\Zpa[y]}
\newcommand \ZpaXn{\Zpa[\Xn]}

\newcommand \gk {\mathbf{k}}
\newcommand \gkb {\ov\gk}
\newcommand \gl {\mathbf{l}}
\newcommand \gA {\mathbf{A}}
\newcommand \gB {\mathbf{B}}
\newcommand \gC {\mathbf{C}}
\newcommand \gD {\mathbf{D}}
\newcommand \gd {\mathbf{d}}
\newcommand \gDd {\mathbf{Dd}}
\newcommand \gE {\mathbf{E}}
\newcommand \gF {\mathbf{F}}
\newcommand \gG {\mathbf{G}}
\newcommand \gI {\mathbf{I}}
\newcommand \gIo {\mathbf{Io}}
\newcommand \gK {\mathbf{K}}
\newcommand \gKb {\ov\gK}
\newcommand \gAh {\widehat\gA}
\newcommand \gKh {\widehat\gK}
\newcommand \gKt {\widetilde\gK}
\newcommand \rhe {^{\mathrm{h}}}
\newcommand \rHe {^{\mathrm{H}}}
\newcommand \rNe {^{\mathrm{N}}}
\newcommand \rhs {^{\mathrm{hs}}}
\newcommand \Khe {\gK\rhe}
\newcommand \KHe {\gK\rHe}
\newcommand \KNe {\gK\rNe}
\newcommand \Kxi {\gK[\xi]}
\newcommand \gKp {{\gK'}}
\newcommand \gKw {\widetilde\gK}
\newcommand \gL {\mathbf{L}}
\newcommand \gLb {\ov\gL}
\newcommand \gLh {\widehat\gL}
\newcommand \gLw {\widetilde\gL}
\newcommand \Lhe {\gL\rhe}
\newcommand \gM {\mathbf{M}}
\newcommand \gP {\mathbf{P}}
\newcommand \gR {\mathbf{R}}
\newcommand \gRa {\mathbf{R}_\mathrm{a}}
\newcommand \gS {\mathbf{S}}
\newcommand \gT {\mathbf{T}}
\newcommand \gV {\mathbf{V}}
\newcommand \gVb {\ov{\gV}}
\newcommand \Vxi {\gV[\xi]}
\newcommand \Vhe {\gV\rhe}
\newcommand \VHe {\gV\rHe}
\newcommand \VNe {\gV\rNe}
\newcommand \fmhe {\fm\rhe}
\newcommand \fmHe {\fm\rHe}
\newcommand \mxi {\fm[\xi]}
\newcommand \mhe {\fmhe}
\newcommand \mHe {\fmHe}
\newcommand \fmNe {\fm\rNe}
\newcommand \Ahe {\gA\!\rhe}
\newcommand \fmhs {\fm\rhs}
\newcommand \Ahs {\gA\rhs}

\newcommand \Af {\gA_f}
\newcommand \Am {(\gA,\fm)}

\newcommand \fmA {\fm_\gA}
\newcommand \fmB {\fm_\gB}

\newcommand \KV {(\gK,\gV)}
\newcommand \LW {(\gL,\gW)}

\newcommand \gVh {\widehat\gV}
\newcommand \fmVh {\widehat{\fm_\gV}}
\newcommand \fmAh {\widehat{\fm_\gA}}
\newcommand \gAt {\widetilde\gA}
\newcommand \gVt {\widetilde\gV}
\newcommand \fmti {\widetilde\fm}
\newcommand \gVp {{\gV'}}
\newcommand \gW {\mathbf{W}}
\newcommand \gX {\mathbf{X}}
\newcommand \gZ {\mathbf{Z}}

\newcommand \Gai {{\Gamma_\infty}}
\newcommand \Gat {\wh{\Gamma}}
\newcommand \Ksep {{\gK^\mathrm{sep}}}
\newcommand \ksep {{\gk^\mathrm{sep}}}
\newcommand \Kac {{\gK^\mathrm{ac}}}
\newcommand \Vsep {{\gV^\mathrm{sep}}}
\newcommand \fmsep {{\fm^\mathrm{sep}}}
\newcommand \vsep {{v_\mathrm{sep}}}
\newcommand \Vac {{\gV^\mathrm{ac}}}
\newcommand \fmac {{\fm^\mathrm{ac}}}

\newdimen\xyrowsp
\xyrowsp=3pt
\newcommand{\SCO}[6]{
\xymatrix @R = \xyrowsp {
                                  &1 \ar@{-}[dl] \ar@{-}[dr] \\
#3 \ar@{-}[ddr]                   &   & #6 \ar@{-}[ddl] \\
                                  &\bullet\ar@{-}[d] \\
                                  &\bullet   \\
#2 \ar@{-}[ddr] \ar@{-}[uur]      &   & #5 \ar@{-}[ddl] \ar@{-}[uul] \\
                                  &\bullet \ar@{-}[d] \\
                                  &\bullet  \\
#1 \ar@{-}[uur]                   &   & #4 \ar@{-}[uul] \\
                                  & 0 \ar@{-}[ul] \ar@{-}[ur] \\
}
}

\newcommand\Pnv[9]{
$$\quad\quad\quad\quad
\vcenter{\xymatrix@C=1.5cm
{
#1 \ar[d]_{#2} \ar[dr]^{#3} \\
{#4} \ar@{-->}[r]_{{#5}\,!}   & {#6} \\
}}
\quad\quad
\vcenter{
\hbox{\small {#7}}
\hbox{~\\[1mm] ~}
\hbox{\small {#8}}
\hbox{~\\[-1mm] ~}
\hbox{\small {#9}}
\hbox{~\\[0mm] ~ }}
$$
}

\newcommand \Adj {\MA{\mathrm{Adj}}}
\newcommand \adj {\MA{\mathrm{adj}}}
\newcommand \Adu {\MA{\mathrm{Adu}}}
\newcommand \Ann {\mathrm{Ann}}
\newcommand \Atom {\mathrm{Atom}}
\newcommand \Aut {\MA{\mathrm{Aut}}}
\newcommand \BZ {\MA{\mathrm{BZ}}}
\newcommand \car {\MA{\mathrm{car}}}
\newcommand \Cl {\MA{\mathrm{Cl}}}
\newcommand \ClW {\MA{\mathrm{Cl}_{\mathrm{W}}}}
\newcommand \Coker {\MA{\mathrm{Coker}}}
\newcommand \Cont{\mathrm{Co}}
\newcommand \Dc {\MA{\mathrm{Dc}}}
\newcommand \DDiv {\MA{\mathrm{Dv}}}
\renewcommand \det {\MA{\mathrm{d\acute{e}t}}}
\renewcommand \deg {\MA{\mathrm{deg}}}
\newcommand \Diag {\MA{\mathrm{Diag}}}
\newcommand \disc {\MA{\mathrm{disc}}}
\newcommand \Disc {\MA{\mathrm{Disc}}}
\newcommand \Div {\MA{\mathrm{Div}}}
\newcommand \DivA {\Div\gA }
\newcommand \DivAp {(\Div\gA)^{+} }
\newcommand \DivB {\Div\gB }
\newcommand \DivBp {(\Div\gB)^{+} }
\newcommand \DkM {\MA{\mathrm{DkM}}}
\newcommand \dv {\MA{\mathrm{div}} }
\newcommand \dvA {\dv_\gA }
\newcommand \dvB {\dv_\gB }
\newcommand \ev {{\mathrm{ev}}}
\newcommand \End {\MA{\mathrm{End}}}
\newcommand \fsac {\MA{\mathrm{fsa}}}
\newcommand \Fix {\MA{\mathrm{Fix}}}
\newcommand \Frac {\MA{\mathrm{Frac}}}
\newcommand \Gal {\MA{\mathrm{Gal}}}
\newcommand \Gfr {\MA{\mathrm{Gfr}}}
\newcommand \Gr {\MA{\mathrm{Gr}}}
\newcommand \gr {\MA{\mathrm{gr}}}
\newcommand \Gram {\MA{\mathrm{Gram}}}
\newcommand \gram {\MA{\mathrm{gram}}}
\newcommand \Grl {\MA{\mathrm{Grl}}}
\newcommand \hauteur {\mathrm{hauteur}}
\newcommand \Hom {\MA{\mathrm{Hom}}}
\newcommand \Id {\MA{\mathrm{Id}}}
\newcommand \Iv {\MA{\mathrm{Iv}}}
\newcommand \I {\mathrm{I}}
\newcommand \Idif {\MA{\mathrm{Idif}}}
\newcommand \Idv {\MA{\mathrm{Idv}}}
\newcommand \Ifr {\MA{\mathrm{Ifr}}}
\newcommand \Icl {\MA{\mathrm{Icl}}}
\renewcommand \Im {\MA{\mathrm{Im}}}
\newcommand \Inf {\MA{\mathrm{Inf}}}
\newcommand \Itf {\MA{\mathrm{Itf}}}
\newcommand \Ker {\MA{\mathrm{Ker}}}
\newcommand \Lst {\MA{\mathrm{Lst}}}
\newcommand \LIN {\mathrm{Lin}}
\newcommand \Lsf {\MA{\mathrm{Lsf}}}
\newcommand \Mat {\MA{\mathrm{Mat}}}
\newcommand \Mip {\mathrm{Min}}
\newcommand \md {\mathrm{md}}
\newcommand \Mgcd {\MA{\mathrm{Mgcd}}}
\renewcommand \mod {\;\mathrm{mod}\;}
\newcommand \Mor {\MA{\mathrm{Mor}}}
\newcommand \NDc {\MA{\mathrm{NDc}}}
\newcommand \poids {\mathrm{poids}}
\newcommand \poles {\hbox {\rm p\^oles}}
\newcommand \pgcd {\MA{\mathrm{pgcd}}}
\newcommand \ppcm {\MA{\mathrm{ppcm}}}
\newcommand \Rad {\MA{\mathrm{Rad}}}
\newcommand \Reg {\MA{\mathrm{Reg}}}
\newcommand \rg{\MA{\mathrm{rg}}}
\newcommand \rgst {\mathrm{rgst}}
\newcommand \Res {\mathrm{Res}}
\newcommand \Rs {\MA{\mathrm{Rs}}}
\newcommand \rPr{\MA{\mathrm{Pr}}}
\newcommand \Rv {\mathrm{Rv}}
\newcommand \Sli {\MA{\mathrm{Sli}}}
\newcommand \Som {\MA{\mathrm{Som}}}
\newcommand \Sup {\MA{\mathrm{Sup}}}
\newcommand \Sace {\MA{\mathrm{Sace}}}
\newcommand \Smtf {\MA{\mathrm{Smtf}}}
\newcommand \Stp {\MA{\mathrm{Stp}}}
\newcommand \St {\mathrm{St}}
\newcommand \Tri {\MA{\mathrm{Tri}}}
\newcommand \Tor {\MA{\mathrm{Tor}}}
\newcommand \tr {\MA{\mathrm{tr}}}
\newcommand \Tr {\MA{\mathrm{Tr}}}
\newcommand \Tsc {\MA{\mathrm{Tsch}}}
\newcommand \Um {\MA{\mathrm{Um}}}
\newcommand \val {\MA{\mathrm{val}}}

\newcommand \SIPD {\MA{\mathrm{SIPD}}}
\newcommand \ARC {\MA{\mathrm{ARC}}}
\newcommand \AFR {\MA{\mathrm{AFR}}}
\newcommand \AFRRV {\MA{\mathrm{AFRRV}}}
\newcommand \AFRNZ {\MA{\mathrm{AFRNZ}}}
\newcommand \PPM {\MA{\mathrm{PPM}}}
\newcommand \PB {\MA{\mathrm{PB}}}

\newcommand \Suslin{{\rm Suslin}}
\newcommand \DAbul {\rD_{\!\Abul}}

\newcommand\MA[1]{\mathop{#1}\nolimits}

\newcommand \sfa {\mathsf{a}}
\newcommand \sfb {\mathsf{b}}
\newcommand \sfc {\mathsf{c}}
\newcommand \sfd {\mathsf{d}}
\newcommand \sfe {\mathsf{e}}
\newcommand \sff {\mathsf{f}}
\newcommand \sfx {\mathsf{x}}
\newcommand \sfy {\mathsf{y}}
\newcommand \sfz {\mathsf{z}}
\newcommand \sbv {\tsbf{v}}
\newcommand \sbw {\tsbf{w}}
\newcommand \sfv {\mathsf{v}}
\newcommand \sfu {\mathsf{u}}
\newcommand \sft {\mathsf{t}}
\newcommand \sfw {\mathsf{w}}

\newcommand \Cdim {\MA{\mathsf{Cdim}}}
\newcommand \Divdim {\MA{\mathsf{Divdim}}}
\newcommand \Glo {\MA{\mathsf{Glo}}}
\newcommand \GK {\MA{\mathsf{GK}}}
\newcommand \GKO {\MA{\mathsf{GK}_0}}
\newcommand \HO {\MA{\mathsf{H}_0}}
\newcommand \HOp {\MA{\mathsf{H}_0^+}}
\newcommand \Hdim {\MA{\mathsf{Hdim}}}
\newcommand \HeA {{\Heit\gA}}
\newcommand \Heit {\MA{\mathsf{Heit}}}
\newcommand \Hspec {\MA{\mathsf{Hspec}}}
\newcommand \Jdim {\MA{\mathsf{Jdim}}}
\newcommand \jdim {\MA{\mathsf{jdim}}}
\newcommand \Jspec {\MA{\mathsf{Jspec}}}
\newcommand \jspec {\MA{\mathsf{jspec}}}
\newcommand \KO {\MA{\mathsf{K}_0}}
\newcommand \KOp {\MA{\mathsf{K}_0^+}}
\newcommand \KTO {\wi{\mathsf{K}}_0}
\newcommand \Kdim {\MA{\mathsf{Kdim}}}
\newcommand \Lin {\mathsf{L}}
\newcommand \Max {\MA{\mathsf{Max}}}
\newcommand \Min {\MA{\mathsf{Min}}}
\newcommand \OQC {\MA{\mathsf{Oqc}}}
\newcommand \Pic {\MA{\mathsf{Pic}}}
\newcommand \Reel {\MA{\mathsf{Reel}}}
\newcommand \Spec {\MA{\mathsf{Spec}}}
\newcommand \Speclin {\MA{\mathsf{Speclin}}}
\newcommand \Spv {\MA{\mathsf{Spv}}}
\newcommand \Spev {\MA{\mathsf{Spev}}}
\newcommand \Sper {\MA{\mathsf{Sper}}}
\newcommand \Val {\MA{\mathsf{Val}}}
\newcommand \Valp {\MA{\mathsf{Val'}}}
\newcommand \SpecA {\Spec\gA}
\newcommand \SpevA {\Spev\gA}
\newcommand \SpvA {\Spv\gA}
\newcommand \SperA {\Sper\gA}
\newcommand \SpecT {\Spec\gT}
\newcommand \Zar {\MA{\mathsf{Zar}}}
\newcommand \ZF {\MA{\mathsf{ZF}}}
\newcommand \ValA {{\Val\gA}}
\newcommand \ValpA {{\Valp\gA}}
\newcommand \ZarA {{\Zar\gA}}

\newcommand \cA {{\cal A}}
\newcommand \cB {{\cal B}}
\newcommand \cC {{\cal C}}
\newcommand \cD {{\cal D}}
\newcommand \cI {{\cal I}}
\newcommand \cJ {{\cal J}}
\newcommand \cF {{\cal F}}
\newcommand \cH {{\cal H}}
\newcommand \cK {{\cal K}}
\newcommand \cL {{\cal L}}
\newcommand \cM {{\cal M}}
\newcommand \cN {{\cal N}}
\newcommand \cP {{\cal P}}
\newcommand \cQ {{\cal Q}}
\newcommand \cR {{\cal R}}
\newcommand \cS {{\cal S}}
\newcommand \cT {{\cal T}}
\newcommand \cV {{\cal V}}

\newcommand \ccd{\mathcal{CD}}
\newcommand \cco{\mathcal{CO}}

\newcommand \Cin{C^{\infty}}

\newcommand \SK {\cS^\rK}
\newcommand \IK {\cI^\rK}
\newcommand \JK {\cJ^\rK}
\newcommand \IH {\cI\rHe}
\newcommand \JH {\cJ\rHe}

\newcommand{\Dpp}[2]{{{\partial #1}\over{\partial #2}}}

\newcommand \J {\Jac}
\newcommand \JJ {\JAC}
\newcommand \JAC {\mathrm{JAC}}
\newcommand \Jac {\mathrm{Jac}}
\newcommand \rja {\mathrm{Ja}}
\newcommand \rc {\mathrm{c}}
\newcommand \Df {\MA{\mathrm{Df}}}
\newcommand \Dfr {\MA{\mathrm{Df}}^\mathrm{R}}
\newcommand \Dfmc {\MA{\mathrm{Dfmc}}}
\newcommand \rd {\mathrm{d}}
\newcommand \rv {\mathrm{v}}
\newcommand \rI {\mathrm{I}}
\newcommand \In {{\rI_n}}

\newcommand \Ic {\mathrm{Ic}}
\newcommand \rC {\mathrm{C}}
\newcommand \rD {\mathrm{D}}
\newcommand \rF {\mathrm{F}}
\newcommand \rG {\mathrm{G}}
\newcommand \rH {\mathrm{H}}
\newcommand \rJ {\mathrm{J}}
\newcommand \Li {\MA{\mathrm{Li}}}
\newcommand \rK {\mathrm{K}}
\newcommand \rL {\mathrm{L}}
\newcommand \Mc {\mathrm{Mc}}
\newcommand \rN {\mathrm{N}}
\newcommand \rP {\mathrm{P}}
\newcommand \rR {\mathrm{R}}
\newcommand \rmSa {\MA{\mathrm{Sa}}}
\newcommand \rmSamc {\MA{\mathrm{Samc}}}
\newcommand \rS {\mathrm{S}}
\newcommand \rU {\mathrm{U}}
\newcommand \rV {\mathrm{V}}
\newcommand \DA {\rD_{\!\gA}}
\newcommand \JA {\rJ_\gA}
\newcommand \JT {\rJ_\gT}

\newcommand\fa{\mathfrak{a}}
\newcommand\fb{\mathfrak{b}}
\newcommand\fc{\mathfrak{c}}
\newcommand\fA{\mathfrak{A}}
\newcommand\fB{\mathfrak{B}}
\newcommand\fD{\mathfrak{D}}
\newcommand\fI{\mathfrak{i}}
\newcommand\fII{\mathfrak{I}}
\newcommand\fj{\mathfrak{j}}
\newcommand\fJ{\mathfrak{J}}
\newcommand\fF{\mathfrak{F}}
\newcommand\ff{\mathfrak{f}}
\newcommand\ffg{\mathfrak{g}}
\newcommand\fG{\mathfrak{G}}
\newcommand\fh{\mathfrak{h}}
\newcommand\fl{\mathfrak{l}}
\newcommand\fm{\mathfrak{m}}
\newcommand\mV{{\fm_\gV}}
\newcommand\mW{{\fm_\gW}}
\newcommand\fM{\mathfrak{M}}
\newcommand\fN{\mathfrak{N}}
\newcommand\fp{\mathfrak{p}}
\newcommand\fP{\mathfrak{P}}
\newcommand\fq{\mathfrak{q}}
\newcommand\fU{\mathfrak{U}}
\newcommand\fV{\mathfrak{V}}
\newcommand\fx{\mathfrak{x}}
\newcommand\fy{\mathfrak{y}}

\newcommand \scC{\mathscr{C}}
\newcommand \scR{\mathscr{R}}

\newcommand{\bma}{\bm{a}}
\newcommand{\bmb}{\bm{b}}
\newcommand{\bmc}{\bm{c}}
\newcommand{\bmd}{\bm{d}}
\newcommand{\bme}{\bm{e}}
\newcommand{\bmf}{\bm{f}}
\newcommand{\bmu}{\bm{u}}
\newcommand{\bmv}{\bm{v}}
\newcommand{\bmw}{\bm{w}}
\newcommand{\bmy}{\bm{y}}
\newcommand{\bmx}{\bm{x}}
\newcommand{\bmz}{\bm{z}}

\newcommand \LLPO {\tsbf{LLPO}}
\newcommand \LPO  {\tsbf{LPO}}

\newcommand \Zg {{\Z[G]}}

\newcommand \vu {\vee} 
\newcommand \vi {\wedge} 
\newcommand \Vu {\bigvee}
\newcommand \Vi {\bigwedge}
\newcommand \im {\rightarrow} 
\newcommand \da {\,\downarrow\!}

\newcommand \vdu[1] {\vdash^{#1}}
\newcommand \vdb[1] {\vdash_{#1}}

\newcommand \Vrai {\mathsf{Vrai}}
\newcommand \Faux {\mathsf{Faux}}
\newcommand \Un {\mathbf{1}}
\newcommand \Deux {\mathbf{2}}
\newcommand \Trois {\mathbf{3}}
\newcommand \Quatre {\mathbf{4}}
\newcommand \Cinq {\mathbf{5}}

\newcommand \una {{\underline{a}}}
\newcommand \ual {{\underline{\alpha}}}
\newcommand \ua  {{\underline{a}}}
\newcommand \ub  {{\underline{b}}}
\newcommand \ube {{\underline{\beta}}}
\newcommand \uc  {{\underline{c}}}
\newcommand \ud  {{\underline{d}}}
\newcommand \udel{{\underline{\delta}}}
\newcommand \ue  {{\underline{e}}}
\newcommand \uf  {{\underline{f}}}
\newcommand \uF  {{\underline{F}}}
\newcommand \ug  {{\underline{g}}}
\newcommand \uh  {{\underline{h}}}
\newcommand \uga {{\underline{\gamma}}}
\newcommand \uP  {{\underline{P}}}
\newcommand \ur {{\underline{r}}}
\newcommand \ut {{\underline{t}}}
\newcommand \uu {{\underline{u}}}
\newcommand \ux {{\underline{x}}}
\newcommand \uxi {{\underline{\xi}}}
\newcommand \uX {{\underline{X}}}
\newcommand \uy {{\underline{y}}}
\newcommand \uY  {{\underline{Y}}}
\newcommand \uz{{\underline{z}}}
\newcommand \uZ{{\underline{Z}}}
\newcommand \uze {{\underline{0}}}
\newcommand \uvep {{\underline{\vep}}}

\newcommand \ak {a_1,\ldots,a_k}
\newcommand \am {a_1,\ldots,a_m}
\newcommand \an {a_1,\ldots,a_n}
\newcommand \aq {a_1,\ldots,a_q}
\newcommand \aln {\alpha_1,\ldots,\alpha_n}
\newcommand \bn {b_1,\ldots,b_n}
\newcommand \bzn {b_0,\ldots,b_n}
\newcommand \bbm {b_1,\ldots,b_m}
\newcommand \ck {c_1,\ldots,c_k}
\newcommand \rcr {c_1,\ldots,c_r}
\newcommand \cq {c_1,\ldots,c_q}
\newcommand \gan {\gamma_1,\ldots,\gamma_n}
\newcommand \un {u_1,\ldots,u_n}
\newcommand \xk {x_1,\ldots,x_k}
\newcommand \Xk {X_1,\ldots,X_k}
\newcommand \xm {x_1,\ldots,x_m}
\newcommand \Xm {X_1,\ldots,X_m}
\newcommand \fn {f_1,\ldots,f_n}
\newcommand \lfm {f_1,\ldots,f_m}
\newcommand \Fn {F_1,\ldots,F_n}
\newcommand \lFm {F_1,\ldots,F_m}
\newcommand \Qpac {\wi\QQ_{p,\mathrm{alg}}}
\newcommand \Qpal {\QQ_{p,\mathrm{alg}}}
\newcommand \sn {s_1,\ldots,s_n}
\newcommand \xn {x_1,\ldots,x_n}
\newcommand \xzn {x_0,\ldots,x_n}
\newcommand \xhn {x_0:\ldots:x_n}
\newcommand \Xn {X_1,\ldots,X_n}
\newcommand \xr {x_1,\ldots,x_r}
\newcommand \Xr {X_1,\ldots,X_r}
\newcommand \xin {\xi_1,\ldots,\xi_n}
\newcommand \xizn {\xi_0,\ldots,\xi_n}
\newcommand \xihn {\xi_0:\ldots:\xi_n}
\newcommand \ym {y_1,\ldots,y_m}
\newcommand \yr {y_1,\ldots,y_r}
\newcommand \Yr {Y_1,\ldots,Y_r}
\newcommand \Yn {Y_1,\ldots,Y_n}
\newcommand \Ym {Y_1,\ldots,Y_m}
\newcommand \yn {y_1,\ldots,y_n}

\newcommand \AT {\gA[T]}
\newcommand \AX {\gA[X]}
\newcommand \Ax {\gA[x]}
\newcommand \AuX {\gA[\uX]}
\newcommand \Aux {\gA[\ux]}
\newcommand \ArX {\gA\lra X}
\newcommand \Axn {\gA[\xn]}
\newcommand \AXn {\gA[\Xn]}

\newcommand \AXm {\gA[\Xm]}
\newcommand \KKXm {{\KK[\Xm]}}
\newcommand \KKuX {{\KK[\uX]}}

\newcommand \AY {\gA[Y]}
\newcommand \Ayn {\gA[\yn]}

\newcommand \BuX {\gB[\uX]}
\newcommand \BuY {\gB[\uY]}
\newcommand \BX {{\gB[X]}}
\newcommand \BT {{\gB[T]}}
\newcommand \BY {{\gB[Y]}}
\newcommand \Bxn {\gB[\xn]}
\newcommand \BXn {{\gB[\Xn]}}
\newcommand \BYm {\gB[\Ym]}

\newcommand \CuX {\gC[\uX]}
\newcommand \CT  {\gC[T]} 
\newcommand \CX  {\gC[X]} 
\newcommand \CXn {\gC[\Xn]} 

\newcommand \kX {{\gk[X]}}
\newcommand \KX {\gK[X]}
\newcommand \Kx {\gK[x]}
\newcommand \VX {\gV[X]}
\newcommand \Vx {\gV[x]}
\newcommand \VT {\gV[T]}
\newcommand \KT {\gK[T]}
\newcommand \KuX {\gK[\uX]}
\newcommand \kuX {\gk[\uX]}
\newcommand \VuX {\gV[\uX]}
\newcommand \Vuxi {\gV[\uxi]}
\newcommand \Vux {\gV[\ux]}
\newcommand \Kux {\gK[\ux]}
\newcommand \kux {\gk[\ux]}
\newcommand \KXk {\gK[\Xk]}
\newcommand \KXm {\gK[\Xm]}
\newcommand \KXn {\gK[\Xn]}
\newcommand \kxm {\gk[\xm]}
\newcommand \kxn {\gk[\xn]}
\newcommand \Ky {\gK[y]}
\newcommand \Vy {\gV[y]}
\newcommand \Kal {\gK[\alpha]}
\newcommand \Wal {\gV[\alpha]}
\newcommand \kY {\gk[Y]}
\newcommand \KY {\gK[Y]}
\newcommand \Kz {\gK[z]}
\newcommand \KZ {\gK[Z]}
\newcommand \kZ {\gk[Z]}
\newcommand \KYn {\gK[\Yn]}
\newcommand \KYm {\gK[\Ym]}
\newcommand \kXn {\gk[\Xn]}
\newcommand \kXr {\gk[\Xr]}
\newcommand \KXr {\gK[\Xr]}
\newcommand \Kxr {\gK[\xr]}

\newcommand \Vxn {\gV[\xn]}
\newcommand \VXn {\gV[\Xn]}

\newcommand \KuY {\gK[\uY]}
\newcommand \Kuy {\gK[\uy]}
\newcommand \Kyn {\gK[\yn]}
\newcommand \Kyr {\gK[\yr]}
\newcommand \kYr {\gk[\Yr]}
\newcommand \KYr {\gK[\Yr]}
\newcommand \Kxn {\gK[\xn]}

\newcommand \LuX {\gL[\uX]}
\newcommand \lXn {\gl[\Xn]}
\newcommand \lxn {\gl[\xn]}
\newcommand \LXn {\gL[\Xn]}
\newcommand \lXr {\gl[\Xr]}
\newcommand \LXr {\gL[\Xr]}
\newcommand \lYr {\gl[\Yr]}
\newcommand \LYr {\gL[\Yr]}

\newcommand \QQXn {\QQ[\Xn]}

\newcommand \Rx {\gR[x]}
\newcommand \RX {\gR[X]}
\newcommand \Rux {\gR[\ux]}
\newcommand \RuX {\gR[\uX]}
\newcommand \RXk {{\gR[\Xk]}}
\newcommand \RXm {{\gR[\Xm]}}
\newcommand \Rxm {{\gR[\xm]}}
\newcommand \RXn {{\gR[\Xn]}}
\newcommand \Rxn {{\gR[\xn]}}
\newcommand \RXzn {{\gR[\Xzn]}}
\newcommand \Rxzn {{\gR[\xzn]}}
\newcommand \RXr {\gR[\Xr]}

\newcommand \Ruy {\gR[\uy]}
\newcommand \Ryn {{\gR[\yn]}}
\newcommand \RRX {\RR[X]}
\newcommand \RRXn {\RR[\Xn]}
\newcommand \RRxn {\RR[\xn]}
\newcommand \RYr {\gR[\Yr]}
\newcommand \RRuX {\RR[\uX]}
\newcommand \RRux {\RR[\ux]}

\newcommand \ZZXn {\ZZ[\Xn]}
\newcommand \ZG {\ZZ[G]}
\newcommand \ZZX {\ZZ[X]}
\newcommand \QQX {\QQ[X]}
\newcommand \CCX {\CC[X]}
\newcommand \ZZx {\ZZ[x]}
\newcommand \QQx {\QQ[x]}
\newcommand \ZZxi {\ZZ[\xi]}
\newcommand \QQxi {\QQ[\xi]}

\newcommand \RRXk {{\RR[\Xk]}}
\newcommand \RRxk {{\RR[\xk]}}
\newcommand \RRXm {{\RR[\Xm]}}
\newcommand \RRxm {{\RR[\xm]}}

\newcommand \lfs {f_1,\ldots,f_s}
\newcommand \lfn {f_1,\ldots,f_n}

\newcommand \Gn  {\gG_n}
\newcommand \Gnk {\gG^n_{k}}
\newcommand \Gnr {\gG^n_{r}}
\newcommand \cGn {\cG_n}
\newcommand \cGnk{\cG_{n,k}}
\newcommand \GGn {\GG^n}
\newcommand \GGnk{\GGn_{k}} 
\newcommand \GGnr{\GGn_{r}}
\newcommand \GA  {\mathbb{GA}}
\newcommand \GAn {\GA^n}  
\newcommand \GAq {\GA^q}
\newcommand \GAnk{\GAn_{k}}
\newcommand \GAnr{\GAn_{r}}
\newcommand \GL {\mathbb{GL}}
\newcommand \GLn {{\GL_n}}
\newcommand \SL {\mathbb{SL}}
\newcommand \SLn {{\SL_n}}
\newcommand \EE {\mathbb{E}}
\newcommand \En {\EE_n}
\newcommand \Pn {\PP^n}
\newcommand \An {\AA^n}
\newcommand \Sl {\mathbf{SL}}
\newcommand \Sln {{\Sl_n}}

\newcommand \Mm {\MM_{m}}
\newcommand \Mn {\MM_{n}}
\newcommand \Mk {\MM_{k}}
\newcommand \Mq {\MM_{q}}
\newcommand \Mr {\MM_{r}}
\newcommand \MMn {\MM_{n}}

\newcommand \Pf {{{\cal P}_{\mathrm{f}}}}
\newcommand \Pfe {{\rm P}_{{\rm fe}}}

\newcommand\hsz{\\ }
\newcommand\hsu{\\ \hspace*{4mm}}
\newcommand\hsd{\\ \hspace*{8mm}}
\newcommand\hst{\\ \hspace*{1,2cm}}
\newcommand\hsq{\\ \hspace*{1,6cm}}
\newcommand\hsc{\\ \hspace*{2cm}}
\newcommand\hsix{\\ \hspace*{2,4cm}}
\newcommand\hsept{\\ \hspace*{2,8cm}}

\newcommand \DK{\mathscr{D}_{\rm K}} 
\newcommand \DN{\mathscr{D}_{\rm N}}
\newcommand \DD{\mathscr{D}_{\rm D}}
\newcommand \vepstar{\varepsilon^\star}
\newcommand \TrBA{\Tr\nolimits_{\gB/\gA}}
\newcommand \trace{\text{\rm trace}}
\renewcommand \norm{\text{\rm N}}
\renewcommand \disc{\text{\rm D}}
\newcommand \normBA{\norm_{\gB/\gA}}
\newcommand \DiscBA{\Disc(\gB/\gA)}
\newcommand \OmegaBA{\Omega_{\gB/\gA}}
\newcommand \swap{\mathrm {swap}}
\renewcommand \Lin{\mathrm {L}}
\newcommand \BoB{\gB\otimes_\gA\gB}
\newcommand \cercle[1]{\tikz[baseline=(char.base)]{
            \node[shape=circle,draw,inner sep = 0.4pt] (char){#1};}}
\newcommand \Bez {\mathscr{B}} 

\title{Identités algébriques permettant  de démontrer  qu'une algèbre libre finie nette est traciquement étale}
\author{Claude Quitté, Henri Lombardi}
\maketitle

L'article suivant est essentiellement dû à Claude Quitté, le deuxième auteur ne s'est occupé que de la mise en forme du texte.

\tableofcontents

\section {Introduction}

L'objectif central de cet article est de donner une démonstration \elr du \thref{NetteImpliesTracicEtale} suivant, dont nous ne connaissons pas de trace dans la littérature existante. Comme  indiqué dans le titre, notre démonstration est basée sur des \idas. Cela confirme l'adage implicite selon lequel une grande partie de l'\alg commutative la plus abstraite se concentre dans des \idas concernant les matrices de \pols sur un anneau commutatif arbitraire.

\begin {theorem} [une algèbre libre de rang fini qui est nette est traciquement étale]
\label{NetteImpliesTracicEtale}

Soit $\gB/\gA$ une algèbre (commutative) libre de rang fini.  Notons
$\OmegaBA$ le $\gB$-module des différentielles de~$\gB/\gA$,
$\normBA : \gB \to \gA$ la norme de $\gB/\gA$ et $\DiscBA$ le
discriminant d'une base de $\gB/\gA$ (défini au carré d'un
inversible près).

\begin {enumerate}
\item
Le $\gB$-module $\OmegaBA$ est de présentation finie et son idéal de
Fitting $\cF_0\big(\OmegaBA\big)$ (idéal de $\gB$), possède la
propriété suivante
$$
\fbox{$b \in \cF_0\big(\OmegaBA\big)  \implies
\DiscBA \ \mid\ \normBA(b)$\,.}
$$
\item
Supposons $\OmegaBA = 0$.  Alors $\DiscBA$ est inversible.
\end {enumerate}
\end {theorem}

\Note La référence de base pour les \diles de Kähler est le livre \cite{Kunz}. 
Dans les ouvrages \cite{CACM} et \cite{ACMC}, le point \textsl{2.} du \thref{NetteImpliesTracicEtale} est démontré uniquement pour le cas où $\gA$ est un \cdi. Cet ouvrage décrit de manière complètement \cov de nombreuses bases de l'\alg commutative et donne en particulier les détails \ncrs pour la compréhension des \idfs et celle du module des \diles de Kähler et son traitement au moyen des matrices bezoutiennes. 
Les notions des syzygies triviales et de relateurs triviaux sont abordées dans la section IV-2. 
\eoe

\medskip Il est clair que le second point du \thref{NetteImpliesTracicEtale} est une conséquence directe du premier car
dans ce cas, $\cF_0\big(\OmegaBA\big) = \gB$ et, en appliquant
la propriété du point \textsl{1.} à $b=1$, on obtient l'inversibilité de
$\DiscBA$.

\smallskip Un corollaire facile du \tho consiste à le généraliser en affaiblissant l'hypothèse.

\begin{corollary} \label{corNetteImpliesTracicEtale}
Considérons une \Alg $\gB$ qui est un \Amo \ptf.  
Alors les conclusions du \thref{corNetteImpliesTracicEtale} sont valables pour $\gB/\gA$.
 \end{corollary}
%
\begin{proof}
En effet après \lon en des \eco $\sn$ de $\gA$, l'\alg $\gB$ devient libre de rang fini, et les conclusions, parce qu'elles sont satisfaites dans chacune des $\gA_{s_i}$-\alg $\gB_i=\gA_{s_i}\otimes_\gA\gB$, sont satisfaites dans $\gB$.

\end{proof}

\smallskip On fera découler le point \textsl{1.} (du
\thref{NetteImpliesTracicEtale}) de l'inclusion $\cF_0(\OmegaBA) \subseteq \mu(\Ann(J))$ donnée dans le lemme~\ref{KahlerSubsetNoether}, qui suppose seulement que $\gB$ est une \Alg \pf,
et de l'identité matricielle donnée dans le point \textsl{2.} du
\thref{AnnJ-TracicIdentity} spécifique au cas où $\gB/\gA$ est
libre.

\section{Jacobien versus bezoutien}\label{secJacBez}

En ce qui concerne la démonstration du premier point du \thref{NetteImpliesTracicEtale}, on commence par étudier le cadre plus général
d'une algèbre $\gB/\gA$ de présentation finie pour laquelle on adopte
les notations suivantes. On fixe une présentation finie de $\gB/\gA$
en $n$ indéterminées $\uX = (X_1, \dots, X_n)$:
$$
\gB = \gA[\ux] = \gA[x_1, \dots, x_n] \simeq \gA[\uX] / I(\uX)
\quad \text{où $I(\uX)$ est un idéal de type fini de $\gA[\uX]$}
$$
On s'alloue deux autres jeux de $n$ indéterminées $\uY=(Y_1, \cdots, Y_n)$ et
$\uZ=(Z_1, \cdots, Z_n)$ de manière à \og relever \fg{}
les déterminants jacobiens en $\uX$ en des
déterminants bezoutiens en $(\uY,\uZ)$, cf. ci-après.

\smallskip

De manière plus structurelle, on présente la $\gA$-algèbre
$\BoB$ de la manière suivante:
$$
\BoB = \gA[\uy,\uz] = \gA[y_1,\dots,y_n, z_1,\dots,z_n]
\simeq \dfrac {\gA[\uY,\uZ]}{I(\uY) + I(\uZ)}
$$
la correspondance étant la suivante:
$$
f(\ux) \otimes g(\ux)  \enspace\longleftrightarrow\enspace f(\uy)g(\uz)
$$
Le morphisme surjectif $\mu \colon \BoB \twoheadrightarrow \gB$ défini par la multiplication n'est autre que l'évaluation $(\uY,\uZ) := \uX$
$$
\left\{\begin {array} {rcl}
\gA[\uY,\uZ] &\longrightarrow &\gA[\uX]\\[2mm]
f(\uY,\uZ)    &\longmapsto     &f(\uX,\uX)
\end {array}\right. 
\qquad\qquad
\mu :
\left\{\begin {array} {rcl}
\gA[\uy,\uz] &\longrightarrow &\gA[\ux]\\[2mm]
f(\uy,\uz)    &\longmapsto     &f(\ux,\ux)
\end {array}\right. 
$$
Notons $J = \ker\mu$. C'est l'idéal de $\BoB$ engendré par les
$b\otimes 1 - 1\otimes b$ pour $b \in \gB$. C'est aussi un
sous-\Bmo pour la structure de \Bmo à gauche $\BoB$.
Même chose à droite.

Comme $\gB = \gA[x_1, \dots, x_n]$, cet idéal $J$ est de type fini,
engendré par les $x_i\otimes 1 - 1\otimes x_i$.  Dans le modèle
$\gA[\uy,\uz]$, il est $n$-engendré:
$$
J \eqdef   \Ker\mu = \langle y_1-z_1, \dots, y_n-z_n\rangle
$$


Soit $f \in \gA[\uX]$. Il y a $n$ polynômes $(U_{j})_{1 \le j \le n}$ de $\gA[\uY,\uZ]$ tels que
$$
f(\uY) - f(\uZ) = \som_{j=1}^n (Y_j-Z_j) U_j(\uY,\uZ)
$$
Il y a un choix quasi-canonique pour les $U_j$. Pour $n=2$:
$$
\begin {array}{ccccc}
f(Y_1,Y_2) - f(Z_1,Z_2) &=& f(Y_1,Y_2) - f(Z_1,Y_2) &+& f(Z_1,Y_2) - f(Z_1,Z_2)
\\[2mm]
&=& (Y_1-Z_1) \frac{f(Y_1,Y_2) - f(Z_1,Y_2)}{Y_1-Z_1} &+&
(Y_2 - Z_2) \frac{f(Z_1,Y_2) - f(Z_1,Z_2)}{Y_2-Z_2}
\end {array}
$$
Et pour $n$ quelconque:
$$
U_j =
\dfrac{f(Z_1,\dots, Z_{j-1}, \cercle{$Y_j$}, Y_{j+1},\dots,Y_n) -
f(Z_1,\dots,Z_{j-1},\cercle{$Z_j$},Y_{j+1},\dots,Y_n)}
{Y_j - Z_j}
$$
Cela étant, soit $(\uF) = (F_1, \dots, F_n)$ un système de $n$ polynômes de $\gA[\uX]$.
On lui associe le \textsl{déterminant bezoutien} $\delta_\uF \in \BoB$ 
en écrivant:
$$
F_i(\uY) - F_i(\uZ) = \som_{j=1}^n (Y_j-Z_j) U_{ij}(\uY,\uZ)
$$
et en évaluant le déterminant bezoutien correspondant en $(\uy,\uz)$:
$$
\delta_\uF := \det (\Bez_{\uY,\uZ}(\uF))(\uy,\uz)
\qquad \text{où} \qquad
\Bez_{\uY,\uZ}(\uF) = (U_{ij})_{1\le i,j \le n} \in \MM_n(\gA[\uY,\uZ])
$$
Alors $\mu(\delta_\uF)$ est le déterminant jacobien de $\uF$ vu dans $\gB$:
$$
\mu(\delta_\uF) := \det (\Jac_\uX(\uF))(\ux)
\qquad \text{où} \qquad
\Jac_\uX(\uF) = (\partial F_i/\partial X_j)_{1\le i,j \le n} \in \MM_n(\gA[\uX])
$$

On verra dans la section \ref{DifferentesSection} que les deux
idéaux de~$\gB$ qui interviennent dans cette inclusion,
$\cF_0(\OmegaBA)$ d'une part, $\mu\big(\Ann(J)\big)$ sont
des \textsl{différentes} de~$\gB/\gA$. Le premier est la différente
$\DK(\gB/\gA)$ de Kälher, le second la différente $\DN(\gB/\gA)$
de Noether. On y prouvera que l'idéal~$J$ de $\BoB$ est un
$\BoB$-module de présentation finie, et que son
$0$-Fitting~$\cF_0(J)$, (un idéal de $\BoB$), a la propriété suivante
$$
\mu(\cF_0(J)) = \cF_0(\OmegaBA).
$$
Puisque $\cF_0(M) \subseteq \Ann(M)$ pour tout module $M$ de
présentation finie, en particulier $\cF_0(J) \subseteq \Ann(J)$, cela
fournira une \textsl{autre} démonstration du lemme \ref{KahlerSubsetNoether}.

\smallskip Mais ici nul besoin de cette section \ref{DifferentesSection} car la
démonstration donnée ci-dessous est élémentaire: elle utilise simplement que
$\widetilde U\, U = \det(U)\Id_n$ pour une matrice carrée
\hbox{$U \in \MM_n$}.

\begin {lemma}
\label{KahlerSubsetNoether}

Dans l'anneau $\BoB$, considérons  l'annulateur $\Ann(J)$ de l'idéal $J$.
Il est caractérisé par:
$$
\Ann(J) = \big\{\,\beta \in \BoB \mid
(b\otimes 1)\beta = (1\otimes b)\beta \enspace \forall\ b \in \gB \,\big\}.
$$
Dans le modèle $\BoB = \gA[\uy,\uz]$ on a 
$$
\Ann(J) = \big\{\,f(\uy,\uz)  \mid f(\uY,\uZ) \in \gA[\uY,\uZ] \;\mathrm{ et }\;
f(\uy,\uz)\times (y_i-z_i) = 0 \enspace \forall\ i=1, \dots, n\,\big\}.
$$
On a alors l'inclusion suivante d'idéaux de $\gB$:
$$
\cF_0(\OmegaBA) \enspace\subseteq\enspace \mu\big(\Ann(J)\big).
$$
\end {lemma}

\begin {proof} 

L'idéal $\cF_0(\OmegaBA)$ est engendré par les évaluations en $\ux$
des déterminants jacobiens $\det(\Jac_\uX(\uF))$ lorsque $(\uF) =
(F_1, \dots, F_n)$ parcourt les $n$-systèmes de polynômes de $I(\uX)$.
Il suffit donc de démontrer qu'un tel générateur est dans
$\mu\big(\Ann(J)\big)$.

Avec la notation précédente $U = \Bez_{\uY,\uZ}(\uF)$, écrivons:
$$
\begin {bmatrix} F_1(\uY)-F_1(\uZ)\\\vdots\\ F_n(\uY)-F_n(\uZ) \end {bmatrix}
= U
\begin {bmatrix} Y_1-Z_1\\\vdots\\ Y_n-Z_n \end {bmatrix}
\qquad \text{d'où} \qquad
\det(U) \begin {bmatrix} Y_1-Z_1\\\vdots\\ Y_n-Z_n \end {bmatrix} =
\widetilde {U}
\begin {bmatrix} F_1(\uY)-F_1(\uZ)\\\vdots\\ F_n(\uY)-F_n(\uZ) \end {bmatrix}
$$
En conséquence $\det(U)(Y_i-Z_i) \in I(\uY) + I(\uZ)$. En évaluant
en $(\uy,\uz)$, on obtient dans $\BoB$:
$$
\delta_\uF \times (y_i-z_i) = 0  \qquad \forall\ i=1, \dots, n.
$$
Ainsi $\delta_\uF \in \Ann(J)$, ce qui montre que $\det(\Jac_\uX(\uF))(\ux)$, égal à
$\mu(\delta_\uF)$, appartient à $\mu(\Ann(J))$.
\end {proof}

\section{Une identité tracique lorsque $\gB/\gA$ est libre finie}

\subsection{Le \tho principal}
On note $\mu$ le morphisme  
\[
\mu \colon \BoB \to \gB,\,x\otimes y\mt xy
\]
 et $J=\ker\mu$. 

\begin {theorem} [une identité tracique attachée à $\mu(\Ann(J))$ lorsque $\gB/\gA$ est libre]
\label{AnnJ-TracicIdentity}~\\
On suppose $\gB/\gA$ libre de rang $n$. On considère un \elt arbitraire $\delta\in \Ann(J)$. Nous notons
\begin{itemize}
\item $b=\mu(\delta)$,
\item $\uvep =(\vep_1, \dots, \vep_n)$ une base de $\gB/\gA$,
\item $T_\uvep\eqdef\big(\TrBA(\vep_i\vep_j)\big)_{1\le i,j \le n}$ la matrice tracique de cette base,
\item $\Disc(\uvep) = \det(T_\uvep)$  (c'est le discriminant de $\gB/\gA$ défini à une unité multiplicative près), 
\item  $M_b$ la matrice de multiplication par $b$ dans la base $\uvep$. 
\end{itemize}


\smallskip \noindent 
On écrit $\delta \in \Ann(J) \subseteq \BoB$  de la manière suivante:
$$
\delta = u_1\otimes\vep_1 + \cdots + u_n\otimes\vep_n
\qquad
u_j \in \gB
$$
Notons $P_\uu  = (u_{ij}) \in \MM_n(\gA) $ la matrice dont la
\textsl{ligne} $i$ est celle des coordonnées de $u_i$ dans la base $\uvep$:
$$
u_i = \som_j u_{ij}\vep_j
$$
\begin {enumerate}
\item
Avec $b = \mu(\delta)$, on a les deux égalités matricielles
dans $\MM_n(\gA)$:
$$
M_b = \big(\TrBA(u_i\vep_j)\big)_{1\le i,j \le n} = P_\uu \,T_\uvep
$$
\item
En prenant les déterminants, on obtient
$\normBA(b) = \det(P_\uu )\Disc(\uvep)$.
\end {enumerate}

\noindent En conséquence, $\DiscBA$ divise $\normBA(b)$ pour tout
$b \in \mu(\Ann(J))$.
\end {theorem}

\subsection {Des rappels d'ordre général}

On rappelle des résultats bien connus. 

\smallskip Tout d'abord $\BoB$ est muni de
deux structures de $\gB$-modules, une à gauche, l'autre à droite:
pour $b \in \gB$, $\beta \in \BoB$
$$
b.\beta = (b\otimes 1)\beta,
\qquad\qquad
\beta.b = \beta(1\otimes b)
$$

\begin {fact} ~

\begin {enumerate}
\item
Sur $\Ann(J)$, les deux structures (à gauche et à doite) de $\gB$-module de $\BoB$
coïncident.
De plus pour $\delta \in \Ann(J)$ et $\beta\in\BoB$ on a
$$
\mu(\beta)\,\delta = \delta\,\mu(\beta) = \beta\delta.
$$

\item
L'idéal $J$ de $\BoB$ et le \Bmo $\Ann(J)$ sont stables
par l'involution $\swap : \BoB \to \BoB$ qui réalise $x\otimes y \mapsto
y\otimes x$.
\end {enumerate}
\end {fact}

\begin {proof} 

\textsl{1.}
Pour $b\in \gB$, il faut vérifier que $(b\otimes 1)\delta = (1\otimes b)\delta$,
ce qui est immédiat car $b\otimes 1 - 1\otimes b \in J$.

\noindent L'égalité $\mu(\beta)\,\delta = \beta\delta$ résulte du fait que
$\mu(\beta)\otimes 1 - \beta \in \Ker(\mu) \eqdef   J$.

\smallskip \noindent 
\textsl{2.} La stabilité de $J$ résulte de  $J = \ker\mu$ et de $\mu \circ \swap = \mu$.
Elle implique la stabilité de $\Ann(J)$.
\end {proof}

\subsection {Identités traciques lorsque $\gB/\gA$ est libre de rang $n$}

On suppose dans la suite que $\gB/\gA$ est libre de rang $n$ de base $\uvep = (\uvep_1, \dots, \uvep_n)$.

\begin {lemma}~ 
\label{AnnJtrace}

\begin {enumerate}
\item
On munit $\BoB$ de la structure de \Blg à gauche.\\
Alors pour $\beta = \som_i x_i\otimes y_i \in  \BoB$ on obtient
$$
\Tr\nolimits_{(\BoB)/\gB}(\beta) = \som_i x_i\TrBA(y_i).
$$

\item
Lorsque $\delta = \som_i x_i\otimes y_i$ appartient à $\Ann(J)$, on obtient
$$
\mu(\delta) =   \som_i x_iy_i =
\som_i x_i\TrBA(y_i) = \som_i y_i\TrBA(x_i).
$$
\end {enumerate}
\end {lemma}

\begin {proof} 

\textsl{1.} Il suffit de le vérifier pour $\beta = x\otimes y$. Pour
la structure de $\gB$-module à gauche de $\BoB$, 
$1\otimes \uvep$ est une $\gB$-base. La matrice de la multiplication
par $x\otimes y$ dans cette base $1\otimes\uvep$ est
$xM_y \in \MM_n(\gB)$ où $M_y \in \MM_n(\gA)$ est la
matrice de multiplication par~$y$ dans la base $\uvep$.
D'où:
$$
\Tr\nolimits_{(\BoB)/\gB}(x\otimes y) = \trace(xM_y) =
x\, \trace(M_y) = x\TrBA(y)
$$

\smallskip \noindent 
\textsl{2.}
Soit un $\gB$-module libre $E$, $x \in E$ et $\alpha\in E^\star$.
L'endomorphisme $v$ de $E$ de rang $\le 1$ associé à $(x,\alpha)$ est
celui défini par $y \mapsto \alpha(y)x$; sa trace est $\trace(v)
= \alpha(x)$.

\noindent Regardons la multiplication par $\delta$ sur $\BoB$ comme un
endomorphisme du $\gB$-module à gauche $\BoB$.  Comme $\delta\beta
= \mu(\beta)\delta$ pour $\beta \in \BoB$, on peut appliquer la
remarque précédente à $E = \BoB$, $x = \delta$, $\alpha = \mu$,
l'endomorphisme $v$ étant la multiplication par~$\delta$.  Donc:
$$
\Tr\nolimits_{(\BoB)/\gB}(\delta) = \mu(\delta)
$$
En utilisant le point \textsl{1.}, on obtient ainsi la première égalité du point \textsl{2.}
La seconde égalité s'obtient en appliquant la première à
$\swap(\delta)=\som_i y_i \otimes x_i$ qui appartient également à $\Ann(J)$.
\end {proof}

\begin {lemma} [un critère d'appartenance à $\Ann(J)$] 
\label {inAnnJ}

Soit $t \in \BoB$ que l'on écrit sur la $\gB$-base $1\otimes \uvep$
du $\gB$-module $\BoB$ à gauche:
$$
t = u_1\otimes \vep_1 + \cdots + u_n\otimes\vep_n, \qquad u_i \in \gB.
$$
Notons $c^{(k)}_{ij} \in \gA$ les constantes de structure définies
par $\vep_i\vep_j = \sum_k c^{(k)}_{ij} \vep_k$.

\smallskip \noindent Alors $t \in \Ann(J)$ si et seulement si
$$
u_i\vep_j = \som_k c^{(i)}_{kj} u_k \qquad \forall\ i,j.
$$
\end {lemma}

\begin {proof} 

On va utiliser, d'une part, pour $x_1, \cdots, x_n \in \gB$ 
$$
\som_i x_i \otimes \vep_i = 0   \iff   x_1 = \cdots = x_n = 0,
$$
et d'autre part, $t \in \Ann(J)$ si et seulement si $\vep_j.t = t.\vep_j$ pour tout $j$.

\noindent Montrons, pour $j$ fixé, que
$$
\som_i \Bigl( u_i\vep_j - \som_k c^{(i)}_{kj}u_k\Bigr) \otimes \vep_i =
\vep_j.t - t.\vep_j.
$$
Il en résultera l'équivalence annoncée.

\noindent Il est clair que $\som_i u_i\vep_j \otimes\vep_i = \vep_j.t$. Pour l'autre
somme, utilisons $c^{(i)}_{kj}u_k\otimes \vep_i = u_k\otimes c^{(i)}_{kj}\vep_i$ On obtient
\begin{displaymath}
\som_i \Bigr(\som_k c^{(i)}_{kj}u_k\Bigr) \otimes \vep_i =
\som_k u_k  \otimes \som_i c^{(i)}_{kj} \vep_i =
\som_k u_k  \otimes \vep_k\vep_j = t.\vep_j.
\qedhere
\end{displaymath}
\end {proof}

\begin{proof}[\textsl{Démonstration du théorème \ref{AnnJ-TracicIdentity}}]~

\noindent \textsl{1.} D'après le lemme \ref{AnnJtrace} on a
$b \eqdef   \mu(\delta) =\som_k \TrBA(u_k)\vep_k$.
Donc
\[
\begin{array}{rcl}
b\vep_j &=& \som_k \TrBA(u_k)\vep_k\vep_j = \som_{k,i} \TrBA(u_k) c^{(i)}_{kj}\vep_i
\\[2mm]
&=&
\som_i \TrBA\bigl(\som_k c^{(i)}_{kj}u_k\bigr)\vep_i.
\end {array}
\]
Or $\sum_k c^{(i)}_{kj}u_k = u_i\vep_j$ d'après le
lemme \ref{inAnnJ}.  Il vient $b\vep_j
= \sum_i \TrBA(u_i\vep_j) \vep_i$, d'où l'égalité matricielle $M_b
= \big(\TrBA(u_i\vep_j\big)_{1\le i,j \le n}$. C'est la première égalité
à démontrer dans le  \thref{AnnJ-TracicIdentity}.

\smallskip \noindent La seconde égalité résulte de
\begin{displaymath}
(P_\uu  T_\vep)_{ij} = \som_k u_{ik} \TrBA(\vep_k\vep_j) =
\TrBA\Bigl( \som_k u_{ik} \vep_k\vep_j \Bigr) =
\TrBA(u_i\vep_j).
\qedhere
\end{displaymath}
\end {proof}

\begin{proof}[\textsl{Démonstration du théorème \ref{NetteImpliesTracicEtale}}]~
Il suffit de démontrer le point \textsl{1}.
C'est une conséquence directe de l'inclusion $\cF_0(\OmegaBA) \subseteq \mu(\Ann(J))$ donnée dans le lemme~\ref{KahlerSubsetNoether} et de l'identité matricielle donnée dans le point \textsl{2.} du
\thref{AnnJ-TracicIdentity}.
\end {proof}

\section {Différentes: $\DK(\gB/\gA)$ de Kähler et $\DN(\gB/\gA)$ de Noether}
\label{DifferentesSection}

\subsection{Présentation finie}
\begin {theorem} 
\label{PresentationFinieJ}

Soit $\gB/\gA$ une algèbre de présentation finie, on note $\mu$ l'\Ali de multiplication
\[\mu : \BoB \to \gB,\;x\otimes y \mt xy\]  
et $J=\ker \mu$.
Alors $J$ est un $\BoB$-module de présentation finie.
\end {theorem}

\begin {proof} 

Adoptons les notations de la section \ref{secJacBez} pour présenter
$\gB/\gA$ sous la forme $\gA[\uX]/I(\uX)$. 
Si $\gB=\aqo\AuX{\lfs}$, alors $I(\uX)=\gen{f_1(\uX),\dots,f_s(\uX)}$.
Ceci fournit une
présentation de $(\BoB)/\gA$ sous la forme $\gA[\uY,\uZ]/\big(I(\uY) +
I(\uZ)\big)$ dans laquelle $J = \langle y_1-z_1, \dots,
y_n-z_n\rangle$.

Il s'agit de montrer que le noyau de
la forme linéaire $\pi$:
$$
\pi :
\xymatrix @C=2.5cm{\gA[\uy,\uz]^n \ar[r]^{[y_1-z_1, \dots, y_n-z_n]} & \gA[\uy,\uz]}
$$
est un sous-module de type fini de $\gA[\uy,\uz]^n$.

\medskip

Pour $f \in \gA[\uX]$, définissons $U^f_j \in \gA[\uY,\uZ]$ par
$$
U^f_j(\uY,\uZ) =
\dfrac{f(Z_1,\dots, Z_{j-1}, \cercle{$Y_j$}, Y_{j+1},\dots,Y_n) -
f(Z_1,\dots,Z_{j-1},\cercle{$Z_j$},Y_{j+1},\dots,Y_n)}
{Y_j - Z_j}
$$
de sorte que:
$$
f(\uY) - f(\uZ) = \som_j (Y_j-Z_j) U^f_j(\uY,\uZ).
$$
Posons
$$
u_j^f = U^f_j(\uy,\uz) \in \gA[\uy,\uz],
\qquad
u^f = (u^f_1, \dots, u^f_n) \in \gA[\uy,\uz]^n.
$$
Pour $f \in I(\uX)$, le vecteur $u^f$ est dans $\ker\pi$.

Notons $(e_1, \dots, e_n)$ la base canonique de $\gA[\uy,\uz]^n$.
Nous allons obtenir un système générateur fini de $\ker\pi$ formé tout d'abord par les relateurs triviaux entre les $y_j-z_j$
$$
(y_i-z_i)e_j - (y_j-z_j)e_i, \qquad
1 \le i < j \le n
$$
et ensuite par les  $u^f$ lorsque $f$ décrit un système générateur (fini) de l'idéal
$I(\uX)$. 

\smallskip Soit une relation dans $\gA[\uy,\uz]$ avec $v_j = V_j(\uy,\uz)$  où
$V_j(\uY,\uZ) \in \gA[\uY,\uZ]$:
$$
\som_j (y_j-z_j)v_j = 0
$$
Dans $\gA[\uY,\uZ]$, ceci signifie
$$
\som_j (Y_j-Z_j)V_j(\uY,\uZ) = f(\uY) - g(\uZ) \quad
\text{avec $f(\uY) \in I(\uY)$ et $g(\uZ) \in I(\uZ)$}.
$$
En spécialisant $(\uY,\uZ) := \uX$, on obtient $f(\uX) = g(\uX)$, qui appartient
à $I(\uX)$. D'où
$$
\som_j (Y_j-Z_j)V_j(\uY,\uZ) = f(\uY) - f(\uZ) =
\som_j (Y_j-Z_j) U^f_j(\uY,\uZ).
$$
En notant $W_j(\uY,\uZ) = V_j(\uY,\uZ) - U^f_j(\uY,\uZ)$:
$$
\som_j (Y_j-Z_j)W_j(\uY,\uZ) = 0.
$$
Comme la suite $(Y_1-Z_1, \dots, Y_n-Z_n)$ est 1-sécante\footnote{Une suite est dite 1-sécante, si le module des \syzys pour (l'\id engendré par) cette suite est engendré par les \syzys triviales.} dans
$\gA[\uY,\uZ]$, on obtient que le vecteur $(W_1, \dots, W_n)$
de $\gA[\uY,\uZ]^n$ est combinaison linéaire des relateurs triviaux
entre les $Y_j-Z_j$. Et par suite que le vecteur $(v_1, \dots, v_n)$
est combinaison linéaire de $u^f$ et des relateurs triviaux
entre les $y_j-z_j$.

En conséquence, la réunion de l'ensemble des relateurs triviaux et de l'ensemble de tous les $u^f$ lorsque~$f$ décrit $I(\uX)$ est un système générateur de $\ker\pi$.
On conclut en remarquant que $f \mapsto u^f$ est $\gA$-linéaire en~$f$, donc on peut limiter les $f$ à un système générateur fini de $I(\uX)$.
\end {proof}

\subsection{Deux idéaux \gui{différentes}}

\begin {definition} [les (idéaux) différentes de Kälher et de Noether]

Soit $\gB/\gA$ une algèbre de présentation finie. On note $\mu$ la multiplication (en tant qu'\Ali)
\[
\mu : \BoB \to \gB,\;x\otimes y\mapsto xy
\] 
et $J=\ker\mu$.
On définit deux idéaux de $\gB$, chacun étant nommé \textsl{idéal différente de
$\gB/\gA$}. Le premier de Kähler, le second de Noether:
$$
\DK(\gB/\gA) = \cF_0(\OmegaBA),
\qquad
\DN(\gB/\gA) = \mu(\Ann(J)).
$$
D'après le lemme \ref{KahlerSubsetNoether}, on a l'inclusion $\DK(\gB/\gA) \subseteq  \DN(\gB/\gA)$.
\end {definition}

\Note Voir \cite{Kunz}. La différente de Kähler est définie dans la section~10. L'annexe~G définit les différentes au sens de Noether et de Dedekind\footnote{La différente de Dedekind $\DD(\gB/\gA)$ est définie lorsque $\gB$ est est un \mptf sur $\gA$. C'est l'idéal engendré par  $f'(x)$, où $f(T)$ est le \polcar de (la multiplication par)~$x$ dans le \Amo \ptf $\gB$. Voir à ce sujet dans \cite{CACM} ou \cite{ACMC}, la proposition III-5.10 et le point 5 de la \dfn VI-3.1.}. Dans la proposition~10.17, Kunz démontre que les trois notions d'idéal différente coïncident dans le contexte suivant: l'anneau $\gA$ est \noe, $\gB/\gA$ est finie et localement intersection complète, et $\Frac(\gB)/\Frac(\gA)$ est étale. \eoe

\begin {lemma} 

Dans le contexte de la définition précédente, l'idéal $J$ est un $\BoB$-module de présentation
finie, cf. le théorème \ref{PresentationFinieJ}.
On peut donc considérer son idéal de Fitting $\cF_0(J)$ qui est un idéal de $\BoB$.
Alors:
$$
\DK(\gB/\gA) = \mu(\cF_0(J)).
$$
Comme on a toujours $\cF_0(J) \subseteq \Ann(J)$, on retrouve l'inclusion
$\DK(\gB/\gA)\subseteq \DN(\gB/\gA)$.
\end {lemma}

\begin {proof} 

Soit $\fa$ un idéal d'un anneau commutatif $\gC$ et $\pi : \gC \twoheadrightarrow \gC/\fa$
la surjection canonique. Si $E$ est un $\gC$-module de présentation finie,
alors $E/\fa E$ est un $(\gC/\fa)$-module de présentation finie et:
$$
\cF_0(E/\fa E) = \pi(\cF_0(E))
$$
On applique cette remarque à $\gC = \BoB$, $\fa=J$ de sorte que
$\gC/\fa \simeq \gB$ via $\mu$ i.e. $\pi = \mu$.
On prend comme module $E = J$. On a alors $E/\fa E = J/J^2$, si
bien que:
$$
\pi(\cF_0(J)) = \cF_0(J/J^2).
$$
Pour conclure, on utilise le fait que $J/J^2$ est un modèle pour
$\OmegaBA$.
\end {proof}

\begin {theorem} 

Soit $\gB/\gA$ une algèbre de présentation finie. Le $\gB$-module
$\OmegaBA$ est alors de présentation finie. Supposons qu'il soit
$m$-engendré (si $\gB/\gA$ est $n$-engendré, on peut prendre $m=n$).

\smallskip\noindent  On a les inclusions suivantes d'idéaux de $\gB$:
$$
\big(\DN(\gB/\gA)\big)^m  \enspace\subseteq\enspace
\big(\Ann(\OmegaBA)\big)^m   \enspace\subseteq\enspace
\DK(\gB/\gA)  \enspace\subseteq\enspace
\DN(\gB/\gA)  \enspace\subseteq\enspace
\Ann(\OmegaBA).
$$
En particulier si $m=1$ (c'est le cas lorsque $\gB/\gA$ est monogène),
tous les idéaux ci-dessus sont égaux.
\end {theorem}

\begin {proof} ~

\noindent 
$\blacktriangleright$
Montrons l'inclusion $\DN(\gB/\gA) \subseteq \Ann(\OmegaBA)$ i.e. 
$\mu\big(\Ann(J)\big) \subseteq
\Ann(J/J^2)$. Pour $\delta \in \Ann(J) \subseteq \BoB$,
il suffit de voir que $\mu(\delta).J \subseteq J^2$.  Soit
$\beta \in J$. Alors
$$
\mu(\delta).\beta \eqdef   (\mu(\delta)\otimes 1)\beta
= (\mu(\delta)\otimes 1 - \delta)\beta
$$
Or $\mu(\delta)\otimes 1 - \delta \in J$ donc $\mu(\delta).\beta \in J^2$.

\medskip\noindent 
$\blacktriangleright$
L'inclusion $\big(\Ann(\OmegaBA)\big)^m \subseteq
\DK(\gB/\gA)$ i.e. $\big(\Ann(\OmegaBA)\big)^m \subseteq \cF_0(\OmegaBA)$ résulte,
pour tout module de présentation finie $M$ engendré par $m$ générateurs,
de l'inclusion $\big(\Ann(M)\big)^m \subseteq \cF_0(M)$
\end {proof}

\subsection{Un exemple: l'algèbre $\gB = \gA[x,y,z]$ définie par les relations
$x^2=y^2=z^2$, $xy=xz=yz=0$}


Cet exemple est dû aux auteurs  \cite{ScSt74} cf. page 102.
Il est reproduit dans l'exemple 5.4, de l'article \cite{IyTa}. 

Il s'agit d'un exemple avec une inclusion stricte $\DK(\gB/\gA) \subsetneq \DN(\gB/\gA)$ lorsque
$\gA$ est un corps de caractéristique $\ne 5$. Peu d'explications sont fournies chez les auteurs
mentionnés.

Les polynômes qui définissent $\gB/\gA$ sont les 5 polynômes suivants:
$$
\uF :\quad X^2 - Y^2,\quad X^2 - Z^2,\quad  XY,\quad XZ,\quad YZ
$$
On a $x^3 = x^2x = y^2x = 0$ et plus généralement $\langle x,y,z\rangle^3 = 0$ i.e. 
$\langle X,Y,Z\rangle^3 \subset \langle\uF\rangle$.
L'algèbre $\gB/\gA$ est libre de base $(1,x,y,z,x^2)$.
La jacobienne de $\uF$ est donnée par:
$$
\Jac(\uF) = 
\bordercmatrix[\lbrack\rbrack] {
     &\partial_X &\partial_Y &\partial_Z \cr
     &2X        &-2Y      &0         \cr      
     &2X        &0        &-2Z       \cr 
     &Y         &X        &0         \cr 
     &Z         &0        &X         \cr 
     &0         &Z        &Y         \cr           
}
$$
Ses mineurs d'ordre 3 sont des polynômes homogènes de degré 3 donc nuls modulo
$\uF$. En conséquence:
$$
\DK(\gB/\gA) = 0
$$
On va montrer que $5x^2 \in \DN(\gB/\gA)$.

\medskip Pour travailler avec $\Ann(J) \subseteq \BoB$, on commence par renommer $x,y,z$ en $x_1,x_2,x_3$.
Alors
$$
\BoB = \gA[y_1,y_2,y_3, z_1,z_2,z_3],\qquad
J = \langle y_1-z_1, y_2-z_2, y_3-z_3\rangle
$$
avec
$$
y_1^2=y_2^2=y_3^2,\enspace y_1y_2 = y_1y_3= y_2y_3 = 0, \qquad
z_1^2=z_2^2=z_3^2,\enspace z_1z_2 = z_1z_3= z_2z_3 = 0
$$
Il nous faut trouver des $t\in \BoB$ tels que $t.(y_i-z_i) = 0$ pour $i=1,2,3$.
Puisque l'on veut $\mu(t) \ne 0$, on cherche $t$ homogène de degré 2 en les $y_i,z_i$.
Voici un candidat
$$
t = y_1^2+z_1^2+s = y_2^2+z_2^2+s = y_3^2+z_3^2+s
\qquad\text{avec}\qquad
s = y_1z_1 + y_2z_2 + y_3z_3
$$
Vérifions que $t.(y_i-z_i) = 0$. Pour $i=1$ par exemple, on écrit $t =
(y_1^2 + y_1z_1 + z_1^2) + (y_2z_2 + y_3z_3)$:
$$
(y_1^2 + y_1z_1 + z_1^2)(y_1 - z_1) = y_1^3 - z_1^3 = 0-0 = 0
;\;
(y_2z_2 + y_3z_3)(y_1 - z_1) = 0
\; \text{car} \;
\begin {array}{c} y_2y_1 = y_3y_1=0\\[2mm] z_2z_1 = z_3z_1=0\end{array}
$$
Ainsi $t \in \Ann(J)$ et $\mu(t) = 5x_1^2$. En revenant aux notations $x,y,z$, on a bien montré $5x^2 \in \DN(\gB/\gA)$.

\medskip 
De la même manière, soit $\gB = \gA[x_1, \dots, x_n]$ l'algèbre
définie par les relations $x_1^2 = \dots = x_n^2$ et $x_ix_j = 0$ pour
$i < j$. Alors $\gB/\gA$ est libre de rang $n+2$, de base $(1,
x_1, \dots, x_n, x_1^2)$. Et on a $(n+2)x_1^2 \in \DN(\gB/\gA)$ via
un argument analogue au cas $n=3$.  Soit
$t \in \BoB = \gA[y_1, \dots, y_n, z_1, \dots, z_n]$
défini par:
$$
t = y_i^2 + z_i^2 + \som_{j=1}^n y_jz_j
$$
On a $t \in \Ann(J)$ et $\mu(t) = (n+2)x_1^2$.
Pour $n \ge 3$, on a $\DK(\gB/\gA) = 0$ puisque les $n$-mineurs
de la jacobienne sont des polynômes homogènes de degré $n$
et que la composante homogène de degré $d$ de $\gB$ est
nulle pour~$d \ge 3$.
 
\bibliographystyle{plainnat-fr}

\end{document}